\documentclass[12pt,twoside]{amsart}
\usepackage{amssymb}
\usepackage{verbatim}
\usepackage{amsmath}
\usepackage{bm}
\usepackage{a4wide}
\usepackage[latin1]{inputenc}
\usepackage[T1]{fontenc}
\usepackage{times}
\usepackage{amssymb,latexsym}
\usepackage{enumerate}
\usepackage[stable]{footmisc}
\usepackage{color}
\usepackage[colorlinks,linkcolor=black,citecolor=black, urlcolor=black]{hyperref}

\makeatletter
\newcommand{\sumprime}{\if@display\sideset{}{'}\sum%
            \else\sum'\fi}
\makeatother

\begin{document}

\numberwithin{equation}{section}

\newtheorem{theorem}{Theorem}[section]
\newtheorem{proposition}[theorem]{Proposition}
\newtheorem{conjecture}[theorem]{Conjecture}
\def\theconjecture{\unskip}
\newtheorem{corollary}[theorem]{Corollary}
\newtheorem{lemma}[theorem]{Lemma}
\newtheorem{observation}[theorem]{Observation}
\newtheorem{definition}{Definition}
\numberwithin{definition}{section} 
\newtheorem{remark}{Remark}
\def\theremark{\unskip}
\newtheorem{kl}{Key Lemma}
\def\thekl{\unskip}
\newtheorem{question}{Question}
\def\thequestion{\unskip}
\newtheorem{example}{Example}
\def\theexample{\unskip}
\newtheorem{problem}{Problem}

\thanks{Supported by National Natural Science Foundation of China, No. 12271101.
\newline
Bo-Yong Chen, School of Mathematical Sciences, Fudan University, Shanghai, 200433, China.
\newline
email: boychen@fudan.edu.cn. 
\newline
Yuanpu Xiong, School of Mathematical Sciences, Fudan University, Shanghai, 200433, China.
\newline
email: ypxiong@fudan.edu.cn}



\title{Regularity of the $p-$Bergman kernel}
\author{Bo-Yong Chen and Yuanpu Xiong}

\date{}

\begin{abstract}
We show that the $p-$Bergman kernel $K_p(z)$ on a bounded domain $\Omega$ is of locally $C^{1,1}$ for $p\geq1$.The proof is based on the locally Lipschitz continuity of the off-diagonal $p-$Bergman kernel $K_p(\zeta,z)$ for fixed $\zeta\in \Omega$.  Global irregularity of  $K_p(\zeta,z)$ is presented for some smooth strongly pseudoconvex domains when $p\gg 1$. As an application of the local $C^{1,1}-$regularity, an upper estimate for the Levi form of $\log K_p(z)$ for $1<p<2$ is provided. Under the  condition that the hyperconvexity index of $\Omega$ is positive, we obtain the log-Lipschitz continuity of $p\mapsto{K_p(z)}$ for $1\leq{p}\leq2$.  
\end{abstract}

\maketitle

\tableofcontents

\section{Introduction}
Regularity  in  extremal  problems in real or complex analysis is of classical interests.  Let us mention two typical classes of extremal problems.   The first class is  Green-type functions.  Recall that the Green function for the Laplacian $\Delta$ is given by 
\[
G_\Omega(\zeta,z):=\sup\left\{u(\zeta): u\in{SH^-(\Omega)},\ \limsup_{z'\rightarrow{z}}(u(z')-E(z'-z))<\infty\right\},
\]
where $SH^-(\Omega)$ denotes the set of all negative subharmonic (sh) functions on a bounded domain $\Omega\subset \mathbb R^n$,  and $E(\cdot)$ is the fundamental solution of $\Delta$.  The well-known theorem of Perron shows that $G_\Omega(\cdot,z)$ is harmonic on $\Omega\setminus \{z\}$; in particular,  it is real-analytic ($C^\omega$) there.   The $p-$Green function $G_{\Omega,p}(\zeta,z)$ is defined in a similar way as $G_\Omega(\zeta,z)$   with the role of  $\Delta$ replaced by the (non-linear) $p$-Laplacian:
\[
\Delta_p u:=\mathrm{div}(|\nabla{u}|^{p-2}\nabla{u}),\ \ \ 1\leq{p}<\infty.
\]
For $p>1$ and $p\neq2$,   Ural'ceva \cite{Uralceva} proved that $G_{\Omega,p}(\cdot,z)$ is  of locally $C^{1,\alpha}$ on $\Omega\setminus\{z\}$ for some $\alpha>0$ (see also Uhlenbeck \cite{Uhlenbeck} and Evans \cite{Evans}). For planar domains, the regularity can be improved beyond $C^2$ (cf. Iwaniec-Manfredi \cite{IM}). The pluricomplex Green function on a bounded domain $\Omega\subset \mathbb C^n$ is defined by
\[
g_\Omega(\zeta,z):=\sup\left\{u(\zeta): u\in PSH^-(\Omega),  \limsup_{z'\rightarrow z} (u(z')-\log|z'-z|)<\infty\right\},
\]
where $PSH^-(\Omega)$ denotes the set of all negative plurisubharmonic (psh) functions on $\Omega$.  Through the work of Guan \cite{Guan} and B\l ocki \cite{BlockiRegularity},  it is known that $g_\Omega(\cdot,z)$ is $C^{1,1}$ in $\overline{\Omega}\setminus \{z\}$ if $\Omega$ is strongly pseudoconvex. On the other hand, an example of Bedford-Demailly \cite{BD} shows that this result is optimal.  Nevertheless, it is unclear whether\/ {\it interior} $C^{1,1}-$regularity is the best possible.  The analysis of  regularity of $G_{\Omega,p}(\cdot,z)$ or $g_\Omega(\cdot,z)$  is nonlinear and nontrivial.  

The second class is Bergman-type functions.  Recall that the kernel of S. Bergman on a bounded domain $\Omega\subset \mathbb C^n$ is given by
\begin{equation}\label{eq:K_2_extremal}
K_\Omega(z):=\sup\left\{|f(z)|^2: f\in{A^2(\Omega)},\ \|f\|_2=1\right\},
\end{equation}
where 
$
A^2(\Omega):=\{f\in \mathcal O(\Omega): \|f\|_2^2:=\int_\Omega |f|^2<\infty\}
$
is the Bergman space, and $\mathcal{O}(\Omega)$ denotes the set of holomorphic functions. Actually, the original definition of the Bergman kernel is  $K_\Omega(z)=\sum_j |h_j(z)|^2$ where $\{h_j\}$ is an/any   orthonormal basis of $A^2(\Omega)$,  from which \eqref{eq:K_2_extremal} and 
$C^\omega-$regularity of $K_\Omega(z)$ immediately follows.  Let us briefly show how to derive $C^\omega-$regularity directly from \eqref{eq:K_2_extremal} by using the {\it Calculus of Variations}.  Following Bergman, we consider the minimizing integral
$$
m_\Omega(z):=\inf\left\{\|f\|_2: f\in A^2(\Omega),f(z)=1\right\}.
$$ 
Clearly,  $K_\Omega(z)=m_\Omega(z)^{-2}$.  Let $m_\Omega(\cdot,z)\in A^2(\Omega)$ be the unique minimizer, which exists by a simple normal family argument.  Given $h\in A^2(\Omega)$ with $h(0)=0$,  set 
 $f_t:= m_\Omega(\cdot,z) +th$.  Since $\|f_t\|_2$ attains its minimum at $t=0$,  it follows that
  $$
  0=\left.\frac{\partial \|f_t\|_2^2}{\partial t}\right|_{t=0} = \int_\Omega h\cdot\overline{m_\Omega(\cdot,z)}.
  $$ 
For any  $f\in A^2(\Omega)$,  we substitute $h=f-f(z)$ into the previous equality and get 
 $$
  \int_\Omega f\cdot\overline{ m_\Omega(\cdot,z)}=f(z) \int_\Omega \overline{ m_\Omega(\cdot,z)}.
$$
  Take $f=m_\Omega(\cdot,z)$,  we get $\int_\Omega \overline{ m_\Omega(\cdot,z) }=m_\Omega(z)^2=1/K_\Omega(z)$,  so that 
 $$
  f(z)=  \int_\Omega f\cdot\overline{K_\Omega(z) m_\Omega(\cdot,z)}=:  \int_\Omega f\cdot\overline{K_\Omega(\cdot,z)}.  
  $$ 
  This is exactly the reproducing formula  of $A^2(\Omega)$,  which implies $K_\Omega(\zeta,z)=\overline{K_\Omega(z,\zeta)}$, so that $K_\Omega(\zeta,z)$ is holomorphic in $(\zeta,\overline{z})$ in view of Hartogs' theorem on separate holomorphicity. Since $K_\Omega(z,z)=K_\Omega(z)m_\Omega(z,z)=K_\Omega(z)$,  it follows that $K_\Omega(z)\in C^\omega(\Omega)$.
  
This argument above essentially hints a basic way of studying the  $p-$Bergman kernel given by
 $$
K_p(z):=\sup\{|f(z)|^p: f\in A^p(\Omega),  \|f\|_p=1\},
$$
where $A^p(\Omega)=\{f\in \mathcal O(\Omega):\|f\|_p^p=\int_\Omega |f|^p<\infty\}$ is the $p-$Bergman space.   Actually,  a parallel argument (see \cite{CZ}, Theorem 2.13) yields the following fundamental reproducing formula of $A^p(\Omega)$ for $p\ge 1$:
\begin{equation}\label{eq:RPF}
f(z) 
 =   \int_\Omega |m_p(\cdot,z)|^{p-2}\,\overline{K_p(z)m_p(\cdot,z)}\,f,  \ \ \ \forall\,f\in A^p(\Omega)
  \end{equation}
  where $m_p(\cdot,z)$ is the unique solution of the following minimizing problem:
\begin{equation}\label{eq:MinProb}
m_p(z)=\inf\left\{\|f\|_p:f\in A^p(\Omega),f(z)=1\right\}.
\end{equation}
It is then reasonable to call $K_p(\zeta,z):=K_p(z)m_p(\zeta,z)$ the off-diagonal $p-$Bergman kernel.  Note that $K_2(\zeta,z)=K_\Omega(\zeta,z)$ and $K_2(z)=K_\Omega(z)$.  
The nonlinear factor $|m_p(\cdot,z)|^{p-2}$ in \eqref{eq:RPF} indicates the basic difference between the case $p\neq 2$ and $p=2$. For example,  as we shall see later,  $K_p(\zeta,z)$ does not coincide with  $\overline{K_p(z,\zeta)}$ for $p\neq 2$, so that we cannot obtain the real analyticity of $K_p(z)$ as the case $p=2$. Actually, there are various domains on which the function $u(z):=K_p(z)-K_2(z)$ violates  the unique continuation principle,  i.e.,  $u$ is not identically zero but it  vanishes  on a nonempty open subset  for some $p>2$;  in particular,  $K_p(z)$ is not real-analytic (cf.  \cite{CZ}). On the other hand, it follows easily from Cauchy's estimates that  $K_p(z)$ is locally Lipschitz continuous. Thus the following question becomes fundamental.

\begin{problem}
What kind of interior regularity does $K_p(\cdot)$  enjoy?
\end{problem}

Let us denote as usual  $C^{k,\alpha}_{\rm loc}(\Omega)$,  $k=0,1,\cdots$,  $0<\alpha\le 1$,     the H\"older space of complex-valued functions $u\in C^k(\Omega)$ whose $k-$th order partial derivatives are locally H\"older continuous with exponent $\alpha$ (locally Lipschitz continuous if $\alpha=1$). We also agree on $C^{0,\alpha}=C^\alpha$ and $C^{0,1}=\mathrm{Lip}$. Throughout this paper, $\Omega$ is always a bounded domain in $\mathbb{C}^n$ unless it is specially emphasized.
Our first main result is the following

\begin{theorem}\label{th:Main}
$K_p(\cdot)\in C^{1,1}_{\rm loc}(\Omega)$\/ for $1\leq{p}<\infty$\footnote{In an earlier version of the paper, the Theorem is stated for $p>1$. We are indebted to the referee for pointing out that the arguments remain valid for $p=1$.}.  Moreover,  we have
\begin{equation}\label{eq:partial derivative}
\frac{\partial{K_p}}{\partial{x}_j}(z)=p\,\mathrm{Re}\,\frac{\partial{K_p(\cdot,z)}}{\partial{x_j}}\bigg|_z\ \ \ \text{and}\ \ \ \frac{\partial{K_p}}{\partial{z}_k}(z)=\frac{p}{2} \cdot \frac{\partial{K_p(\cdot,z)}}{\partial{z_k}}\bigg|_z,
\end{equation}
for $1\le j\le 2n$ and $1\leq{k}\leq{n}$. Here $(x_1,\cdots,x_{2n})$ and $(z_1,\cdots,z_n)$ denote the real and complex coordinates in $\mathbb R^{2n}=\mathbb C^n$ respectively.
\end{theorem} 

\begin{remark}

For saving symbols, we write $\partial{K_p(\cdot,z)}/\partial{z_k}|_z$ instead of $\partial{K_p(\zeta,z)}/\partial{\zeta_k}|_{\zeta=z}$, and so on.

\end{remark}

In view of Rademacher's theorem, the local $C^{1,1}$-regularity of $K_p(\cdot)$ enables us to define the Levi form $i\partial\overline{\partial}\log{K_p}(z;X)$ in the usual sense almost everywhere. This combined with Theorems 5.1 and 5.3  in \cite{CZ} yields
\begin{equation}\label{eq:LeviLower}
i\partial\overline{\partial}\log{K_p}(z;X) \ge \begin{cases}
C(z;X)^2,\ \ \ &1<p<2,\\
B_p(z;X)^2,\ \ \ &2<p<\infty,
\end{cases}
\end{equation}
holds in the usual sense. Here  $C(z;X)$ denotes the Carath\'eodory metric and $B_p(z;X)$ is the $p-$Bergman metric given by
\[
B_p(z;X):=  {K_p(z)^{-\frac1p}}\cdot \sup \left\{|Xf(z)|:f\in A^p(\Omega),f(z)=0,\|f\|_p=1\right\}.
\]

It is natural to expect analogous upper bounds for  $i\partial\overline{\partial}\log{K_p}(z;X)$.  

\begin{theorem}\label{th:Levi_upper_p_le_2}
For $1<p< 2$,  we have
\begin{equation}\label{eq:Levi_upper_0}
i\partial\bar{\partial} \log{K_p}(z;X) \leq \frac{pq}{4} \widehat{B}_p(z;X)^2-(q-1)\frac{|XK_p(z)|^2}{K_p(z)^2},\ \ \ \text{a.e.}\ \ z\in\Omega,
\end{equation}
where $\frac1p+\frac1q=1$
and
\[
\widehat{B}_p(z;X):=K_p(z)^{-1/p}\sup\left\{|Xf(z)|: f\in{A^p(\Omega)},\ \|f\|_p=1\right\}.
\]
\end{theorem}

\begin{remark}
Equality in \eqref{eq:Levi_upper_0} holds for $p=q=2$ (see Proposition     \ref{prop:Equality_p=2}).  
\end{remark}

Regularity of $K_p(\cdot)$ is first considered in \cite{CX}, where local $C^{1,1/2}-$regularity is obtained. A key ingredient in \cite{CX} is the regularity of the function $z\mapsto{K_p(\zeta,z)}$ for fixed $\zeta$, which can be regarded as a problem on  stability of the minimizer in \eqref{eq:MinProb} with respect to the parameter $z$. More precisely, we have $K_p(\cdot)\in C_{\rm loc}^{1,\alpha}(\Omega)$  if one can show that 
for each compact set $S\subset \Omega$ there exists a constant $C$ such that 
\begin{equation}\label{eq:off_diagonal_alpha}
|K_p(\zeta,z)-K_p(\zeta,z')| \leq C |z-z'|^\alpha,\ \ \  \forall\,\zeta,z,z'\in{S}
\end{equation}
(see, e.g., Lemma \ref{lm:C1/2_K_1} below). \eqref{eq:off_diagonal_alpha} has been verified for $\alpha=1/2$ in Theorem 4.7 in \cite{CZ}. Later, the exponent $\alpha$ is improved to $\alpha=1$ for $1<p\le 2$ and $\alpha=p/(2p-2)$ for $p>2$, in an earlier version of the present paper appeared as arXiv: 2302.06877v1. Very recently, Yinji Li, in his undergraduate thesis  \cite{LiYinji}, obtained that \eqref{eq:off_diagonal_alpha} holds with $\alpha=1-\varepsilon$ for any $\varepsilon>0$.
The proof of Li is also based on the argument in \cite{CZ}. The crucial improvement for $p>2$ essentially follows from an elementary consequence of Cauchy's estimates: for any open set $U$ with $S\subset U\subset\subset \Omega$ and any $s>0$,  there exists a constant $C=C_{p,s,S,U}>0$ such that
$$
|\partial{f}(z)|^s \le C \int_U |f|^s,\ \ \ f\in A^p(\Omega),\ z\in S, 
$$
in place of the weaker inequality 
$
|\partial{f}(z)|^p \lesssim \|f\|_p^p
$
used in the above mentioned arXiv paper. With this replacement,  we are able to prove the following.  

\begin{theorem}\label{th:Holder_off_diagonal}
Let $1<p<\infty$ and $S$ a compact subset of $\Omega$. There exists a constant $C>0$ depending on $p$ and $S$ such that
\begin{equation}\label{eq:off-Holder}
|K_p(\zeta,z)-K_p(\zeta,z')| \leq C|z-z'|,\ \ \  \forall\,\zeta,z,z'\in{S}.
\end{equation}
The same conclusion also holds for $m_p(\cdot,\cdot)$.
\end{theorem}

It follows from Theorem \ref{th:Holder_off_diagonal} and Rademacher's theorem that $w\mapsto{K_p(z,w)}$ is differentiable almost everywhere. On the other hand, Theorem \ref{th:Main} yields the following

\begin{corollary}\label{cor:differentiable_z_z}
The function $w\mapsto{K_p(z,w)}$ is differentiable at $w=z$.
\end{corollary}

The case $p=1$ is  important among $p\in [1,\infty)$ since only $K_1$ shares the same transformation rule as $K_2$ under biholomorphic mappings on \textit{arbitrary}  domains (see \cite{CZ},  Proposition 2.7). We have the following observation

\begin{corollary}\label{cor:Image}
The function
$$
z\mapsto\left.\frac{\partial{K_1(z,\cdot)}}{\partial{x_j}}\right|_z
$$
is differentiable and its image lies on the imaginary axis.  
\end{corollary}

Simple example of the unit disc shows that the image of 
$$
z\mapsto\left.\frac{\partial{K_p(z,\cdot)}}{\partial{x_j}}\right|_z
$$
is the whole complex plane for any $p>1$ (see \S\,5.2). This phenomenon may be viewed a basic difference of the $p-$Bergman theory between $p=1$ and $p>1$. 

The global regularity of $K_p(\zeta,z)$ is surprisingly bad in general.  A celebrated  result of Kerzman \cite{Kerzman} states that $K_2(\zeta,z)$ is $C^\infty$ on $\overline{\Omega}\times \overline{\Omega}\setminus\{\zeta=z\}$ if $\Omega$ is a bounded strongly pseudoconvex domain with $C^\infty-$boundary.   For bounded complete Reinhardt domains,   $K_2(\cdot,z)$ is even holomorphic in a neighborhood of $\overline{\Omega}$ for each $z\in \Omega$ (cf.  \cite{BB}).  In sharp contrast with these facts and Theorem \ref{th:Holder_off_diagonal},  we have the following 

\begin{theorem}\label{th:boundary-regularity}
Let $\Omega\subset \mathbb C^n$ be a bounded complete Reinhardt domain with smooth boundary such that $K_2(\zeta,z)$ is not zero-free.  Then the following properties hold:
\begin{enumerate}
\item[$(1)$] For each $\alpha>0$,  there exist $k\in \mathbb Z^+$ and $z\in \Omega$ such that $K_{2k}(\cdot,z)\notin C^\alpha (\overline{\Omega})$.   
\item[$(2)$] For each $\alpha>0$,  there exists $k\in \mathbb Z^+$ such that either $K_{2k}(\cdot,z)\notin C(\overline{\Omega})$ for some $z\in \Omega$ or $K_{2k}(\zeta,\cdot)\notin C^\alpha_{\rm loc}(\Omega)$ for some $\zeta\in \partial{\Omega}$.
\end{enumerate}
Here  $C^\alpha$ means $C^{0,\alpha}$, i.e., H\"older continuous with exponent $\alpha$.    
\end{theorem}

\begin{remark}
There are plenty of bounded  complete Reinhardt,  smooth strongly pseudoconvex (even convex) domains on which $K_2(\zeta,z)$ is not zero-free (cf.  \cite{Boas}).  
\end{remark}

It is also interesting to study the parameter regualrity of $K_p(z)$ or $K_p(\zeta,z)$, i.e., the dependence in $p$.

\begin{theorem}\label{th:LlogL}
Let $1\le p\le 2$.  If the hyperconvexity index $\alpha(\Omega)$ is positive, then 
$$
|K_s(z)-K_p(z)| \lesssim |s-p||\log |s-p||
$$
 as $s\rightarrow p$ in $[1,2]$.
\end{theorem}

Recall that
\[
\alpha(\Omega):=\sup\left\{\alpha\geq0;\ \exists\,\rho\in{PSH^-(\Omega)}\cap{C(\Omega)},\ \text{s.t.}\ -\rho\lesssim \delta_\Omega^\alpha\right\}
\]
(cf. \cite{Chen17}). The class of domains with $\alpha(\Omega)>0$ includes  bounded pseudoconvex domains with Lipschitz boundaries (cf.  \cite{Harrington}). Under the condition that $\alpha(\Omega)>0$, the continuity of $K_p(z)$ at $p=2$ has been verified in \cite{CZ}; nevertheless, there are examples showing that $p\mapsto{K_p(z)}$ might not be continuous at every $p\in(2,+\infty)$ (see \cite{CX}, Proposition 1.5).

Similar as \cite{CZ} and \cite{CX}, our analysis relies heavily on the reproducing formula \eqref{eq:RPF}, which follows from the calculus of variations. In order to handle the nonlinear factor $|m_p(\cdot,z)|^{p-2}$ in \eqref{eq:RPF}, we borrow some techniques from the nonlinear analysis of $\Delta_p$ (cf. \cite{Lindqvist06}). Various complex-analytic tools are also used, for example, semi-continuity property of the log canonical threshold (cf.  \cite{DK}), the Borel-Carath\'{e}odory inequality, as well as $L^2$ estimates for the $\overline{\partial}-$operator.

To conclude this introduction,  we remark that differentiability of $K_p(\cdot)$ for $p>1$ also follows from  the abstract theory on  differentiability of norms in Banach spaces.  A key observation is the fact $K_p(z)^\frac1p = \|L_z\|$, where $L_z(f):=f(z)$ is the evaluation functional on $A^p(\Omega)$. The map $z\mapsto{L_z}$ is weakly holomorphic in the sense that $z\mapsto{L_z(f)=f(z)}$ is a holomorphic function for every $f\in{A^p(\Omega)}$,  which is  strongly holomorphic in the sense of Taylor \cite{Taylor}.  In particular, it is Fr\'{e}chet differentiable.  On the other hand,  since $A^p(\Omega)$ is uniformly convex in view of the Clarkson's inequality,  it follows from a theorem of Lindenstrauss \cite{Lindenstrauss} that the operator norm $\|\cdot\|$ in $(A^p(\Omega))^*$ is Fr\'{e}chet differentiable except at $0$,  which implies that  $K_p(\cdot)$ is differentiable.   But the above method does not give the explicit formula  \eqref{eq:partial derivative}.

\section{Preliminaries}\label{sec:preliminary}

\subsection{Some inequalities}

Let us recall  some "elementary" inequalities for  $a,b\in \mathbb C$,  which arise from the nonlinear analysis of the $p-$Laplacian (cf.  \cite{Lindqvist06},  \S\,12;  see also \cite{CZ},  Appendix): 
\begin{eqnarray}
\mathrm{Re}\left\{(|b|^{p-2}\overline{b}-|a|^{p-2}\overline{a})(b-a)\right\}
&\geq& (p-1)(|a|+|b|)^{p-2}|b-a|^2\nonumber\\
& & +(2-p)|\mathrm{Im}(a\overline{b})|^2(|a|+|b|)^{p-4},\ \ \ 1\leq{p} \leq2,\label{eq:ineq_p_leq2}\\
\mathrm{Re}\left\{(|b|^{p-2}\overline{b}-|a|^{p-2}\overline{a})(b-a)\right\}
&\geq& \frac{1}{2}\left(|a|^{p-2}+|b|^{p-2}\right)|b-a|^2\ \ \ p>2,\label{eq:ineq_p_geq2_1}
\end{eqnarray}
\begin{eqnarray}
|b|^p & \geq &  |a|^p+p\mathrm{Re}\left\{|a|^{p-2}\overline{a}(b-a)\right\}+A_p|b-a|^2(|a|+|b|)^{p-2},\ \ \ 1<p<2,\label{eq:ineq_leq2}\\
|b|^p & \geq & |a|^p+p\mathrm{Re}\left\{|a|^{p-2}\overline{a}(b-a)\right\}+{4^{-p-3}}|b-a|^p,\ \ \ p\ge 2,\label{eq:ineq_geq2}\\
|b| & \geq & |a|+\mathrm{Re}\left\{|a|^{-1}\overline{a}(b-a)\right\}+B_1|\mathrm{Im}(\overline{a}b)|^2(|a|+|b|)^{-3},\label{eq:ineq_q_equal_1}
\end{eqnarray}
\begin{eqnarray}
\left||b|^{p-2}\overline{b}-|a|^{p-2}\overline{a}\right|  & \leq & 2^{2-p}|b-a|^{p-1},\ \ \ 1<p<2,\label{eq:ineq_leq2_2}\\
\left||b|^{p-2}\overline{b}-|a|^{p-2}\overline{a}\right| &\leq&  (p-1)|b-a|\left(|a|+|b-a|\right)^{p-2},\ \ \ p\ge 2.\label{eq:ineq_geq2_2}
\end{eqnarray} 
Here $A_p:=p(p-1)/2$ and $B_1>0$ is a numerical constant. 

 \eqref{eq:ineq_leq2} together with \eqref{eq:ineq_geq2} and \eqref{eq:ineq_q_equal_1}
yield the following
\begin{eqnarray}
|a|^p &\leq& |b|^p+p\mathrm{Re}\left\{|a|^{p-2}\overline{a}(a-b)\right\}\label{eq:ineq_geq1_1}\\
&\leq& |b|^p+p|a|^{p-1}|b-a|,\ \ \ \forall\,1\leq{p}<\infty.\label{eq:ineq_geq1}
\end{eqnarray}
Moreover,  the trivial inequality
\begin{equation}\label{eq:ineq_elementary}
|a+b|^p\leq\max\{1,2^{p-1}\}\left(|a|^p+|b|^p\right),\ \ \ 0<p<\infty,
\end{equation}
will also be frequently used.

We also need the following refinement of \eqref{eq:ineq_geq2},  which might be useful for other purposes. 

\begin{lemma}\label{lm:ineq}
Let $q>2$. For any $a,b\in\mathbb{C}$, we have
\begin{eqnarray}
|b|^q &=& |a|^q + q\,\mathrm{Re}\left\{|a|^{q-2}\overline{a}(b-a)\right\}+\frac{q^2}{4}|a|^{q-2}|b-a|^2\nonumber\\
& & +\frac{q(q-2)}{4}|a|^{q-4}\mathrm{Re}\left\{\overline{a}^2(b-a)^2\right\}+R_q(a,b),\label{eq:ineq_geq4}
\end{eqnarray}
where the remaining term $R_q(a,b)$ is bounded by a sum of functions of the following type:
\[
C(|a|^\alpha+|b-a|^\alpha)|b-a|^\beta
\]
for suitable constants $C>0$, $\alpha\ge 0$ and $\beta>2$ with $\alpha+\beta=q$.
\end{lemma}

The proof is elementary but lengthy. We shall postpone it to the appendix. We write $q$ instead of $p$ in Lemma \ref{lm:ineq} since it is used in the proof of Theorem \ref{th:Levi_upper_p_le_2} with $1<p<2$ and $1/p+1/q=1$.

\subsection{Some notations on derivatives}

We identify $\mathbb{C}^n$ with $\mathbb{R}^{2n}$ by setting $z_j=x_j+ix_{n+j}$.  A real (tangent) vector $v$ of $\mathbb{R}^{2n}$ is given by $v=\sum^{2n}_{j=1}v_j\partial/\partial{x_j}$.  It can be identified with a complex tangent vector $X=\sum^n_{j=1}X_j\partial/\partial{z}_j$ as follows:
$$
X_j=v_j+iv_{n+j}.
$$
We denote by $iv$ the real vector corresponding to $iX$, i.e.,
\[
iv=-\sum^n_{j=1}v_{n+j}\frac{\partial}{\partial{x_j}}+\sum^n_{j=1}v_j\frac{\partial}{\partial{x_{n+j}}}.
\]
In what follows,    $e^{i\theta}v$  means the real vector $\cos\theta\cdot{v}+\sin\theta\cdot{iv}$. The corresponding complex vector is then denoted by $e^{i\theta}X$.

For a $C^1$ function $u$ on a bounded domain $\Omega\subset\mathbb{C}^n=\mathbb{R}^{2n}$ and a real (tangent) vector $v$,  we denote by $\partial{u}/\partial{v}$ the real directional derivative along $v$,  that is,
\[
\frac{\partial{u}}{\partial{v}}(z):=\sum^{2n}_{j=1}v_j\frac{\partial{u}}{\partial{x_j}}(z)=\lim_{\mathbb R\ni t\rightarrow 0} \frac{u(z+tv)-u(t)}{t}.
\]
If we denote the real gradient of $u$ by $\nabla{u}$, then
\[
\frac{\partial{u}}{\partial{v}}(z)=\langle{\nabla{u}(z),v}\rangle,
\]
where $\langle{\cdot,\cdot}\rangle$ is the inner product in $\mathbb{R}^{2n}$. For a complex (tangent) vector $X=\sum^n_{j=1}X_j\partial/\partial{z_j}$,  we define 
\[
Xu(z):=\sum^n_{j=1}X_j\frac{\partial{u}}{\partial{z_j}}(z)\ \ \ \text{and}\ \ \ \overline{X}u:=\sum^n_{j=1}\overline{X}_j\frac{\partial{u}}{\partial\bar{z}_j}(z),
\]
where
\[
\frac{\partial{u}}{\partial{z_j}}=\frac{1}{2}\left(\frac{\partial{u}}{\partial{x_j}}-i\frac{\partial{u}}{\partial{x_{n+j}}}\right),\ \ \ \frac{\partial{u}}{\partial\bar{z}_j}=\frac{1}{2}\left(\frac{\partial{u}}{\partial{x_j}}+i\frac{\partial{u}}{\partial{x_{n+j}}}\right).
\]
If the complex tangent vector  $X$  represents  the real tangent vector $v$,  then 
\[
Xu=\frac{1}{2}\left(\frac{\partial{u}}{\partial{v}}-i\frac{\partial{u}}{\partial(iv)}\right),\ \ \ \overline{X}u=\frac{1}{2}\left(\frac{\partial{u}}{\partial{v}}+i\frac{\partial{u}}{\partial(iv)}\right).
\]
In particular, if $u$ is a holomorphic function, i.e., $\overline{X}u=0$ for all complex vector $X$, then 
\begin{equation}\label{eq:complex_real_derivative_holomophic}
Xu=\frac{\partial{u}}{\partial{v}}.
\end{equation}

 It follows from Rademacher's theorem that if $u\in C^{1,1}_{\rm loc}$ then the Hessian $D^2u$ exists almost everywhere, and is locally bounded on $\Omega$.
Thus we may define the Levi form as follows:
\begin{equation}\label{eq:Levi_def}
i\partial\bar{\partial} u(z;X):=X\overline{X}u=\sum^n_{j,k=1}\frac{\partial^2u}{\partial{z_j}\partial\bar{z}_k}X_j\overline{X}_k,\ \ \ \text{a.e.}\ z\in \Omega,\ \forall\,X.
\end{equation}
For a general psh function $u$  on $\Omega$,  $i\partial\bar{\partial}u$ only exists as a current. More precisely,  $i\partial\bar{\partial} u(\cdot;X)$ is a positive measure.  Set
\[
T_r u(z;X):=\frac{1}{r^2}\left(\frac{1}{2\pi}\int^{2\pi}_0u(z+re^{i\theta}X)d\theta-u(z)\right),\ \ \ r>0.
\]
It is known that $T_ru(\cdot;X)$ converges to $i\partial\bar{\partial} u(\cdot;X)$  in the sense of weak convergence of measures as $r\rightarrow 0+$ (cf. \cite{BT76},  Proposition 6.3). Moreover, if $u$ is a locally $C^{1,1}$  psh function,  then the current $i\partial\bar{\partial}u$ coincides with the Levi form given by \eqref{eq:Levi_def} (see e.g., \cite{Blocki}, Theorem 1.1.11).

\begin{lemma}\label{lm:Levi_upper_estimate}
Let $u$ be a $C^{1,1}_{\rm loc}$ psh function on a domain $\Omega\subset\mathbb{C}^n$.  If there is a continuous function $A=A(z,X)$ on $\Omega\times\mathbb{C}^n$ such that $T_ru(z;X)\leq{A(z,X)}$ for all $z,X,r$,  then
\[
i\partial\bar{\partial}u(z;X)\leq{A(z,X)},\ \ \ \text{a.e.}\ z\in \Omega.
\]
\end{lemma}

\begin{proof}
For any $\phi\in{C_0(\Omega)}$ with $\phi\geq0$,  we have
\[
\int_\Omega\phi\cdot{i\partial\bar{\partial}u(\cdot;X)}= \lim_{r\rightarrow0}\int_\Omega\phi\cdot{T_ru(\cdot;X)}\leq\int_\Omega\phi\cdot{A(\cdot,X)}.
\]
Thus $i\partial\bar{\partial}u(\cdot;X)\leq{A(\cdot,X)}$ as positive measures,  from which the assertion immediately follows.  
\end{proof}

\section{Interior regularity of $K_p(\zeta,z)$}\label{sec:interior_off}

Let $1\le p <\infty$ and $z,z'\in\Omega$.  We have
\begin{eqnarray}\label{eq:H_integral}
& &\int_\Omega\left(|K_p(\cdot,z)|^{p-2}\overline{K_p(\cdot,z)}-|K_p(\cdot,z')|^{p-2}\overline{K_p(\cdot,z')}\right)\big(K_p(\cdot,z)-K_p(\cdot,z')\big)\nonumber\\
&=& \int_\Omega|K_p(\cdot,z)|^p + \int_\Omega|K_p(\cdot,z')|^p -\int_\Omega|K_p(\cdot,z)|^{p-2}\overline{K_p(\cdot,z)}K_p(\cdot,z') \nonumber\\
&& - \int_\Omega|K_p(\cdot,z')|^{p-2}\overline{K_p(\cdot,z')}K_p(\cdot,z)\nonumber\\
&=& K_p(z)^{p-1}+K_p(z')^{p-1}-K_p(z)^{p-2}K_p(z,z')-K_p(z')^{p-2}K_p(z',z)
\end{eqnarray}
in view of \eqref{eq:RPF}.
Set
\begin{equation}\label{eq:H_p}
\widehat{H}_p(z,z'):=K_p(z)^{p-1}+K_p(z')^{p-1}-\mathrm{Re}\left\{K_p(z)^{p-2}K_p(z,z')+K_p(z')^{p-2}K_p(z',z)\right\}.
\end{equation}

\begin{lemma}\label{lm:H_p_Holder}
For any $z,z'\in\Omega$,  we have
\begin{equation}\label{eq:H_p_Holder_0}
\int_\Omega\frac{|\mathrm{Im}(\overline{K_1(\cdot,z)}K_1(\cdot,z'))|^2}{\left(|K_1(\cdot,z)|+|K_1(\cdot,z')|\right)^3} \leq \widehat{H}_1(z,z'),
\end{equation}
\begin{equation}\label{eq:H_p_Holder_1}
\int_\Omega\left(|K_p(\cdot,z)|+|K_p(\cdot,z')|\right)^{p-2}|K_p(\cdot,z)-K_p(\cdot,z')|^2 \leq \frac{\widehat{H}_p(z,z')}{p-1},\ \ \ 1<p\le 2,
\end{equation}
\begin{equation}\label{eq:H_p_Holder_2}
\int_\Omega\left(|K_p(\cdot,z)|^{p-2}+|K_p(\cdot,z')|^{p-2}\right)|K_p(\cdot,z)-K_p(\cdot,z')|^2 \leq 2 \widehat{H}_p(z,z'),\ \ \ p>2,
\end{equation}
\end{lemma}

\begin{proof}
It suffices to substitute $a=K_p(\cdot,z')$ and $b=K_p(\cdot,z)$ into \eqref{eq:ineq_p_leq2} and \eqref{eq:ineq_p_geq2_1},  then take integration over $\Omega$.
\end{proof}

\subsection{The case $1<p\leq2$}
Our start point is the following

\begin{proposition}\label{prop:H_p_upper_p_le_2}
Let $1<p\leq2$. For every compact set $S\subset\Omega$, we have
\[
|\widehat{H}_p(z,z')|\leq{C}|z-z'|\cdot\|K_p(\cdot,z)-K_p(\cdot,z')\|_p,\ \ \ \forall\,z,z'\in{S},
\]
where the constant $C$ depends only on $S$.
\end{proposition}

\begin{proof}
Given $z\in \Omega$,  define  the evaluation functional as follows:
\[
L_z(f):=f(z),\ \ \ f\in{A^p(\Omega)}.
\]
Then we have
\begin{eqnarray}
\widehat{H}_p(z,z') 
&=& \mathrm{Re}\left\{K_p(z)^{p-2}K_p(z,z)-K_p(z)^{p-2}K_p(z,z')\right\}\nonumber\\
& & -\mathrm{Re}\left\{K_p(z')^{p-2}K_p(z',z)-K_p(z')^{p-2}K_p(z',z')\right\}\nonumber\\
&=& \mathrm{Re}\left\{K_p(z)^{p-2}L_z\left(K_p(\cdot,z)-K_p(\cdot,z')\right)\right\}\nonumber\\
& & -\mathrm{Re}\left\{K_p(z')^{p-2}L_{z'}\left(K_p(\cdot,z)-K_p(\cdot,z')\right)\right\}\nonumber\\
&=& \mathrm{Re}\left\{\left(K_p(z)^{p-2}L_z-K_p(z')^{p-2}L_{z'}\right)\left(K_p(\cdot,z)-K_p(\cdot,z')\right)\right\}\label{eq:H_p_decomposition}.
\end{eqnarray}
Note that $L_z$ and $L_{z'}$ are continuous in view of the mean-value inequality.  Moreover, if one fixes some open set $U$ with $S\subset{U}\subset\subset\Omega$, then Cauchy's estimates together with the mean-value inequality give
\begin{equation}\label{eq:H_p_Cauchy}
|L_z(f)-L_{z'}(f)|\leq\sup_U|\partial{f}|\cdot|z-z'|\leq{C_S}|z-z'|\cdot\|f\|_p,
\end{equation}
so that
\[
\|L_z-L_{z'}\| \leq {C_S} |z-z'|,
\]
for $z,z'\in{S}$, where $C_S$ is a generic constant depending only on $S$. 

Since $K_p(\cdot)$ is locally Lipschitz continuous (cf. \cite{CZ}, Proposition 2.11) and
\[
1/|\Omega| \leq K_p(\zeta) \leq C_nd(S,\partial\Omega)^{-2n},\ \ \ \forall\,\zeta\in{S},
\]
we see that
\begin{eqnarray*}
& & \|K_p(z)^{p-2}L_z-K_p(z')^{p-2}L_{z'}\|\\
&\leq& K_p(z)^{p-2}\|L_z-L_{z'}\|+\left|K_p(z)^{p-2}-K_p(z')^{p-2}\right|\cdot\|L_{z'}\|\\
&=& K_p(z)^{p-2}\|L_z-L_{z'}\|+K_p(z')^{1/p}\left|K_p(z)^{p-2}-K_p(z')^{p-2}\right|\\
&\leq & C_S|z-z'|.
\end{eqnarray*}
Thus \eqref{eq:H_p_decomposition} yields
\begin{eqnarray*}
\widehat{H}_p(z,z') 
&\leq& \|K_p(z)^{p-2}L_z-K_p(z')^{p-2}L_{z'}\|\cdot\|K_p(\cdot,z)-K_p(\cdot,z')\|_p\\
&\leq & C_S|z-z'|\cdot\|K_p(\cdot,z)-K_p(\cdot,z')\|_p.
\end{eqnarray*}
\end{proof}

\begin{proof}[Proof of Theorem \ref{th:Holder_off_diagonal} for $1<p\leq2$]
For any $1<p < 2$,  we take $s:=p(2-p)/2$, so that $-2s/p=p-2$ and $2s/(2-p)=p$. By H\"{o}lder's inequality, we have
\begin{eqnarray}\label{eq:Off_Holder}
& &\int_\Omega|K_p(\cdot,z)-K_p(\cdot,z')|^p\\
&=& \int_\Omega\left(|K_p(\cdot,z)|+|K_p(\cdot,z')|\right)^{-s}|K_p(\cdot,z)-K_p(\cdot,z')|^p\left(|K_p(\cdot,z)|+|K_p(\cdot,z')|\right)^s\nonumber\\
&\leq& \left(\int_\Omega\left(|K_p(\cdot,z)|+|K_p(\cdot,z')|\right)^{p-2}|K_p(\cdot,z)-K_p(\cdot,z')|^2\right)^{p/2}\nonumber\\
& &\times\left(\int_\Omega\left(|K_p(\cdot,z)|+|K_p(\cdot,z')|\right)^p\right)^{1-p/2}\nonumber.
\end{eqnarray}
Note that the previous inequality holds trivially for $p=2$.  On the other hand,  
 \eqref{eq:ineq_elementary} gives
\begin{eqnarray*}
\int_\Omega\left(|K_p(\cdot,z)|+|K_p(\cdot,z')|\right)^p &\leq& 2^{p-1}\int_\Omega\left(|K_p(\cdot,z)|^p+|K_p(\cdot,z')|^p\right)\\
&\leq& 2^{p-1}\left(K_p(z)^{p-1}+K_p(z')^{p-1}\right).
\end{eqnarray*}
This together with \eqref{eq:H_p_Holder_1}, \eqref{eq:Off_Holder} and Proposition \ref{prop:H_p_upper_p_le_2} yield 
\[
\|K_p(\cdot,z)-K_p(\cdot,z')\|_p^p\leq  C_S\frac{|z-z'|^{p/2}}{(p-1)^{p/2}}\|K_p(\cdot,z)-K_p(\cdot,z')\|_p^{p/2}
\]
where $C_S$ is a generic constant depending only on  $S$. Thus
\begin{equation}\label{eq:norm-Lip}
\|K_p(\cdot,z)-K_p(\cdot,z')\|_p\leq C_S\frac{|z-z'|}{p-1},
\end{equation}
so that
$$
|K_p(\zeta,z)-K_p(\zeta,z')|\le K_p(\zeta)^{1/p} \|K_p(\cdot,z)-K_p(\cdot,z')\|_p\leq C_S\frac{|z-z'|}{p-1}.
$$
\end{proof}

\subsection{The case $p>2$}

In an early version of this paper (arXiv: 2302.06877v1), we have shown
\begin{equation}\label{eq:Holder_off_diagonal_p_ge_2_weaker}
|K_p(\zeta,z)-K_p(\zeta,z')|\leq C_{p,S}|z-z'|^{p/(2p-2)},\ \ \ \forall\,\zeta\in{S}.
\end{equation}
The exponent $p/(2p-2)$ can be improved to $1$ with the help of the following result due to Li \cite{LiYinji}.

\begin{proposition}\label{prop:H_p_upper_p_ge_2}
Let $p>2$ and $s>0$. Let $S\subset\Omega$ be a compact set in $\Omega$ and $U$ an open set with $S\subset{U}\subset\subset\Omega$. Then we have
\[
|\widehat{H}_p(z,z')|\leq{C}|z-z'|\cdot\left(\int_U|K_p(\cdot,z)-K_p(\cdot,z')|^s\right)^{1/s},\ \ \ \,\forall\,z,z'\in{S},
\]
where the constant $C$ depends on $p,s,S$ and $U$.
\end{proposition}

\begin{proof}
For the sake of completeness, we also include the proof here. Since \eqref{eq:H_p_decomposition} also holds when $p>2$, we may write
\begin{eqnarray*}
\widehat{H}_p(z,z')
&=& \left(K_p(z)^{p-2}-K_p(z')^{p-2}\right)\mathrm{Re}\left\{L_z\left(K_p(\cdot,z)-K_p(\cdot,z')\right)\right\}\\
& & +K_p(z')^{p-2}\mathrm{Re}\left\{(L_z-L_{z'})\left(K_p(\cdot,z)-K_p(\cdot,z')\right)\right\}.
\end{eqnarray*}
Take $U':=\{z\in{U}:d(z,S)<d(z,\partial{U})/2\}$, so that $S\subset{U'}\subset\subset{U}$. It follows from Cauchy's estimates that
\[
\left|L_z\left(K_p(\cdot,z)-K_p(\cdot,z')\right)\right|\leq C_{s,S,U}\left(\int_U|K_p(\cdot,z)-K_p(\cdot,z')|^s\right)^{1/s}
\]
and
\begin{eqnarray*}
\left|(L_z-L_{z'})\left(K_p(\cdot,z)-K_p(\cdot,z')\right)\right|
&\leq& C_{s,S,U}|z-z'|\sup_{U'}|\partial(K_p(\cdot,z)-K_p(\cdot,z'))|\\
&\leq& C_{s,S,U}|z-z'|\left(\int_U|K_p(\cdot,z)-K_p(\cdot,z')|^s\right)^{1/s}.
\end{eqnarray*}
Since $K_p(z)$ is locally Lipschitz continuous, the assertion follows.
\end{proof}

\begin{proof}[Proof of Theorem \ref{th:Holder_off_diagonal} for $p>2$]
Since $\{K_p(\cdot,z):z\in{S}\}$ forms a compact family of holomorphic functions in the topology of locally uniformly convergence (cf.  \cite{CZ},  Proposition 2.11),  it follows from a celebrated theorem due to Demailly-Koll\'{a}r \cite{DK} that 
\[
\int_U|K_p(\cdot,z)|^{-c}\leq{M},\ \ \ \forall\,z\in{S}
\]
for suitable constants $c=c_S>0$ and $M=M_S>0$, where
\[
U:=\{z\in \Omega:d(z,S)<d(S,\partial \Omega)/2\}
\]
satisfies $S\subset{U}\subset\subset\Omega$. Take $s=2c/(p-2+c)<2$, so that $(p-2)s/(2-s)=c$. By H\"{o}lder's inequality, we have
\begin{eqnarray*}
&& \int_U|K_p(\cdot,z)-K_p(\cdot,z')|^s\\
 &\leq& \int_U|K_p(\cdot,z)|^{(p-2)s/2}|K_p(\cdot,z)-K_p(\cdot,z')|^s|K_p(\cdot,z)|^{-(p-2)s/2}\\
&\leq& \left(\int_\Omega|K_p(\cdot,z)|^{p-2}|K_p(\cdot,z)-K_p(\cdot,z')|^2\right)^{s/2}\times\left(\int_U|K_p(\cdot,z)|^{-c}\right)^{1-s/2},
\end{eqnarray*}
so that
\[
\left(\int_U|K_p(\cdot,z)-K_p(\cdot,z')|^s\right)^{2/s} \leq M^{1-s/2} \int_\Omega|K_p(\cdot,z)|^{p-2}|K_p(\cdot,z)-K_p(\cdot,z')|^2.
\]
Interchange the role of $z$ and $z'$,  we obtain 
\[
\left(\int_U|K_p(\cdot,z)-K_p(\cdot,z')|^s\right)^{2/s} \leq M^{1-s/2} \int_\Omega|K_p(\cdot,z')|^{p-2}|K_p(\cdot,z)-K_p(\cdot,z')|^2.
\]
Therefore, 
\begin{eqnarray*}
& & \int_\Omega\left(|K_p(\cdot,z)|^{p-2}+|K_p(\cdot,z')|^{p-2}\right)|K_p(\cdot,z)-K_p(\cdot,z')|^2\\
&\geq& C_{p,S}\left(\int_U|K_p(\cdot,z)-K_p(\cdot,z')|^s\right)^{2/s}
\end{eqnarray*}
for some generic constant $C_{p,S}$. On the other hand, \eqref{eq:H_p_Holder_2} together with Proposition \ref{prop:H_p_upper_p_ge_2} yield that
\begin{eqnarray*}
& & \int_\Omega\left(|K_p(\cdot,z)|^{p-2}+|K_p(\cdot,z')|^{p-2}\right)|K_p(\cdot,z)-K_p(\cdot,z')|^2\\
&\leq& C_{p,S}|z-z'|\left(\int_U|K_p(\cdot,z)-K_p(\cdot,z')|^s\right)^{1/s}.
\end{eqnarray*}
Thus
\[
\left(\int_U|K_p(\cdot,z)-K_p(\cdot,z')|^s\right)^{1/s}\leq C_{p,S}|z-z'|.
\]
The assertion follows immediately from the mean-value inequality.
\end{proof}

\subsection{The case $p=1$}

We only get a weaker conclusion than Theorem \ref{th:Holder_off_diagonal}. 

\begin{theorem}\label{th:Holder_off_diagonal_K_1}
Given $z\in \Omega$,  define  $S_z:=\{K_1(\cdot,z)=0\}$.  For every neighbourhood $U$ of $z$ with $U\subset\subset\Omega\setminus{S_z}$, we have
\[
|K_1(\zeta,z)-K_1(\zeta,z')| \le C |z-z'|,\ \ \ \forall\,\zeta,z'\in{U},
\]
where the constant $C$ depends only on $z$ and $U$. The same conclusion also holds for $m_1(\cdot,\cdot)$.
\end{theorem}

We need the following result, whose proof is completely analogous to that of Proposition \ref{prop:H_p_upper_p_ge_2}.

\begin{proposition}\label{prop:H_p_upper_p=1}
For every open sets $U,U_1$ with $U\subset\subset{U_1}\subset\subset\Omega$, we have
\[
|\widehat{H}_1(z,z')|\leq{C}|z-z'|\cdot\sup_{U_1}|K_1(\cdot,z)-K_1(\cdot,z')|,\ \ \ \,\forall{z,z'\in{U}},
\]
where the constant $C$ depends on $U$ and $U_1$.
\end{proposition}

\begin{proof}[Proof of Theorem \ref{th:Holder_off_diagonal_K_1}]
Take domains $U_1,U_2$ and $U_3$ with $U\subset\subset{U_1}\subset\subset{U_2}\subset\subset{U_3}\subset\subset\Omega\setminus{S_z}$ (note that $\Omega\setminus{S_z}$ is a domain).   By \eqref{eq:H_p_Holder_0} and Proposition \ref{prop:H_p_upper_p=1}, we have
\begin{equation}\label{eq:Holder_K1}
\int_\Omega\frac{|\mathrm{Im}(\overline{K_1(\cdot,z)}K_1(\cdot,z'))|^2}{\left(|K_1(\cdot,z)|+|K_1(\cdot,z')|\right)^3} \leq \widehat{H}_1(z,z') \lesssim |z-z'|\cdot\sup_{U_1}|K_1(\cdot,z)-K_1(\cdot,z')|.
\end{equation}
Here and in what follows,  all implicit constants depend only on $U$ and $U_j$ ($j=1,2,3$) unless they are specially mentioned. There exists a constant $C\gg1$ which depends only on $U,U_3$  such that
\[
C^{-1}\leq|K_1(\zeta,z)|\leq\max\{|K_1(\zeta,z)|,|K_1(\zeta,z')|\}\leq C,\ \ \ \forall\,\zeta\in{U_3}.
\]
Thus
\begin{eqnarray*}
\frac{|\mathrm{Im}(\overline{K_1(\zeta,z)}K_1(\zeta,z'))|^2}{\left(|K_1(\zeta,z)|+|K_1(\zeta,z')|\right)^3} &=& \frac{|K_1(\zeta,z)|^4}{\left(|K_1(\zeta,z)|+|K_1(\zeta,z')|\right)^3}\left|\mathrm{Im}\frac{K_1(\zeta,z')}{K_1(\zeta,z)}\right|^2\\
&\asymp& \left|\mathrm{Im}\frac{K_1(\zeta,z')}{K_1(\zeta,z)}\right|^2,\ \ \ \forall\,\zeta\in{U_3}.
\end{eqnarray*}
By the mean-value inequality,  we have 
\begin{eqnarray}\label{eq:Im_h}
\sup_{U_2}\left|\mathrm{Im}\left\{\frac{K_1(\cdot,z')}{K_1(\cdot,z)}\right\}\right|^2 &\lesssim& \int_{U_3}\left|\mathrm{Im}\left\{\frac{K_1(\cdot,z')}{K_1(\cdot,z)}\right\}\right|^2\nonumber\\
&\lesssim& \int_\Omega\frac{|\mathrm{Im}(\overline{K_1(\cdot,z)}K_1(\cdot,z'))|^2}{\left(|K_1(\cdot,z)|+|K_1(\cdot,z')|\right)^3}\nonumber\\
&\lesssim& |z-z'|\cdot\sup_{U_1}|K_1(\cdot,z)-K_1(\cdot,z')|.
\end{eqnarray}
To proceed the proof,  we need the following celebrated Borel-Carath\'{e}odory inequality (see e.g., \cite{LaxZalcman}, Appendix B): 
\begin{equation}\label{eq:BC}
\sup_{\mathbb{D}_r}|h|\leq\frac{2r}{R-r}\sup_{\mathbb{D}_R}\mathrm{Im}\,h+\frac{R+r}{R-r}|h(0)|,\ \ \ \forall\,h\in\mathcal{O}(\mathbb{D}_R),
\end{equation}
where $0<r<R$ and $\mathbb{D}_r$ denotes the disc with center $0$ and radius $r$.  Apply \eqref{eq:BC} on every complex line through $\zeta\in\Omega$, we obtain
\begin{equation}\label{eq:BC_1}
\sup_{B(\zeta,r/2)}|h|\leq2\sup_{B(\zeta,r)}\mathrm{Im}\,h+3|h(\zeta)|
\end{equation}
for every $h\in\mathcal{O}(\Omega)$ and $0<r<\delta(\zeta)$.  In particular,  if we take 
\[
h:=\frac{K_1(\cdot,z')}{K_1(\cdot,z)}-1,
\]
then
\begin{equation}\label{eq:h_1}
\sup_{U_2}|\mathrm{Im}\,h|\lesssim|z-z'|^{1/2}\cdot\sup_{U_1}|K_1(\cdot,z)-K_1(\cdot,z')|^{1/2}
\end{equation}
in view of \eqref{eq:Im_h}. Moreover, since $K_1(\cdot,z)$ is holomorphic and $K_1(\cdot)$ is locally Lipschitz continuous, we have
\begin{equation}\label{eq:h_2}
|h(z)|\leq\frac{|K_1(z,z')-K_1(z',z')|+|K_1(z')-K_1(z)|}{K_1(z)}\lesssim|z-z'|,\ \ \ z,z'\in{U}.
\end{equation}
Take a chain of balls connecting $z$ and $\zeta\in{U_1}$, we have 
\begin{eqnarray*}
& & |K_1(\zeta,z)-K_1(\zeta,z')| = |K_1(\zeta,z)|\cdot|h(\zeta)|\\
&\lesssim& |z-z'|^{1/2}\cdot\sup_{U_1}|K_1(\cdot,z)-K_1(\cdot,z')|^{1/2}+|z-z'|
\end{eqnarray*}
in view of $\eqref{eq:BC_1}\sim  \eqref{eq:h_2}$, so that $A:=(\sup_{U_1}|K_1(\cdot,z)-K_1(\cdot,z')|)^{1/2}$ satisfies the inequality
\[
A^2\lesssim|z-z'|^{1/2}A+|z-z'|,
\]
from which the assertion immediately follows.
\end{proof}

\section{Global irregularity of $K_p(\zeta,z)$}\label{sec:boundary}
In this section,  we shall prove Theorem \ref{th:boundary-regularity}.  The argument relies heavily on the following 

   \begin{proposition}[cf.  \cite{CZ},  Proposition 3.3]\label{prop:BR}
   Suppose that $\Omega$ is a bounded simply-connected domain in $\mathbb C^n$ and $m_p(\cdot,z)$ is zero-free for some $p\ge 1$ and $z\in \Omega$. Then  
   \begin{enumerate}
   \item[$(1)$] \ \ \ $K_s(z) = K_p(z)$\/ for any $s\ge p$.
   \item[$(2)$] \ \ \ $m_s(\cdot,z) = m_p(\cdot,z)^{p/s}$\/ for any $s\ge p$.
   \end{enumerate}
     \end{proposition}  
     
\begin{proof}[Proof of Theorem \ref{th:boundary-regularity}]
     Following \cite{CZ},  we define
     $$
     \mathcal F(\Omega):=\left\{z\in \Omega: K_2(\cdot,z) \ \text{is\ zero-free} \right\},\  \mathcal N(\Omega):=\Omega\backslash \mathcal F(\Omega).
     $$ 
  It is a simple consequence of  Hurwitz's theorem  that $\mathcal F(\Omega)$ is a closed subset in $\Omega$,  i.e.,  $\mathcal N(\Omega)$ is open in $\Omega$ (cf.  \cite{CZ},  Proposition 3.5).  
  
  Suppose  $\Omega\subset \mathbb C^n$ is a bounded smooth complete Reinhardt domain on which $K_2(\zeta,z)$ is not zero-free,  i.e.,  $\mathcal N(\Omega)\neq \emptyset$.  We also  have  $\mathcal F(\Omega)^\circ \neq \emptyset$ in view of \cite{CZ},  Proposition 3.7.  Thus given 
a point  
$
  z_0\in \partial \mathcal F(\Omega)^\circ \cap \Omega =\partial\overline{\mathcal{N}(\Omega)}\cap\Omega,
  $
  there exists $\varepsilon>0$ such that
  $$
B(z_0,\varepsilon)\cap\mathcal{F}(\Omega)^\circ\neq\emptyset,\ \ \ B(z_0,\varepsilon)\cap\mathcal{N}(\Omega)\neq\emptyset.
$$
Take a sequence of points $\{z_j\}\subset \mathcal N(\Omega)$ such that $z_j\rightarrow z_0$ as $j\rightarrow \infty$.  For each $j$,  there exists a point $\zeta_j\in \Omega$ with 
$$
K_2(\zeta_j,z_j)=0=m_2(\zeta_j,z_j).
$$
Let $\zeta_0$ be a limit point of the sequence $\{\zeta_j\}$ inside $\overline{\Omega}$. 
The power series expansion of $K_2$ implies that 
$$
K_2(\lambda \zeta,z) = K_2(\zeta,\bar{\lambda}z)
$$
for any $\lambda=(\lambda_1,\cdots,\lambda_n)$ for which both sides are defined (cf.  \cite{BB}).   This shows that 
$$
K_2(\zeta,z)=K_2(\lambda^{-1}\zeta,\bar{\lambda}z),  \ \ \ \lambda^{-1}:=(\lambda_1^{-1},\cdots,\lambda_n^{-1}),
$$
 can be extended to a function which is holomorphic in $(\zeta,\bar{z})$ for $\zeta\in U$ and $z\in V$,  where $U$ is a neighborhood of $\overline{\Omega}$ and $V$ is a neighborhood of $z_0$.   
Thus
$$
K_2(\zeta_0,z_0)=0=m_2(\zeta_0,z_0).
$$
On the other hand,  since $\mathcal F(\Omega)$ is closed and $z_0\in\partial\mathcal{F}(\Omega)^\circ\subset\mathcal{F}(\Omega)$,  so $K_2(\cdot,z_0)$ and $m_2(\cdot,z_0)$ are zero-free on $\Omega$.  Thus $\zeta_0\in \partial \Omega$ and Proposition \ref{prop:BR} gives
\begin{equation}\label{eq:irreg}
m_2(\zeta,z_0)=m_{2k}(\zeta,z_0)^k
\end{equation}
for all $\zeta\in \Omega$.   

Let $\mathcal L_0$ be the normal complex  line of $\Omega$ at $\zeta_0$.  Since $m_2(\cdot,z_0)$ is holomorphic on $U\supset \overline{\Omega}$ whose zero set lives outside $\Omega$,  it follows that
$$
k_0:=\mathrm{ord}_{\zeta_0} \left(m_2(\cdot,z_0)|_{\mathcal L_0\cap U}\right)<\infty.
$$ 
We claim that $K_{2k}(\cdot,z_0)\notin C^\alpha(\overline{\Omega})$ for any $\alpha>0$ and $k>k_0/\alpha$.   To see this,  suppose on the contrary that $K_{2k}(\cdot,z_0)\in C^\alpha(\overline{\Omega})$ for some $\alpha>0$ and $k>k_0/\alpha$.  Then \eqref{eq:irreg} remains valid on $\overline{\Omega}$.   But this implies that for $\zeta\in \mathcal L_0\cap \Omega$ sufficiently close to $\zeta_0$,  
$$
|\zeta-\zeta_0|^{k_0}\asymp  |m_2(\zeta,z_0)|  =  |m_{2k}(\zeta,z_0)|^k 
\lesssim |\zeta-\zeta_0|^{\alpha k},
$$
which is absurd.  Thus we have verified the first assertion.

For the second assertion,  we first infer from Proposition \ref{prop:BR} that
$$
m_2(\zeta,z)=m_{2k}(\zeta,z)^k,\ \ \ \zeta\in \Omega,\,z\in \mathcal F(\Omega)^\circ.
$$
Suppose $K_{2k}(\cdot,z)$ is continuous on $\overline{\Omega}$ for any $z\in \Omega$,  so is $m_{2k}(\cdot,z)$. 
Then we have
$$
m_2(\zeta_0,z)=\lim_{j\rightarrow \infty} m_2(\zeta_j,z) = \lim_{j\rightarrow \infty} m_{2k}(\zeta_j,z)^k=m_{2k}(\zeta_0,z)^k,\ \ \ z\in \mathcal F(\Omega)^\circ,
$$
while $m_2(\zeta_0,z_0)=0$.  
Without loss of generality,  we may assume that $z_0\in \partial B\cap \partial \mathcal F(\Omega)^\circ$ for suitable ball $B\subset \mathcal F(\Omega)^\circ$.  Since $\overline{K_2(\zeta_0,\cdot)}=K_2(\cdot,\zeta_0)$ is holomorphic in some ball $B_0\ni z_0$ which is not identically zero,   so there is a complex line $\mathcal L'_0$ which intersects transversally with $\partial B$ at $z_0$,  such that 
$$
k_0':=\mathrm{ord}_{z_0}  m_2(\zeta_0,\cdot)|_{\mathcal L'_0\cap B_0}<\infty.
$$  
Suppose on the contrary that $K_{2k}(\zeta_0,\cdot)\in C^\alpha_{\rm loc}(\Omega)$ for some $\alpha>0$ and $k>k'_0/\alpha$.  Then  for $z\in \mathcal L_0'\cap B$ sufficiently close to $z_0$,  
$$
|z-z_0|^{k_0'}\asymp  |m_2(\zeta_0,z)|  =  |m_{2k}(\zeta_0,z)|^k 
\lesssim |z-z_0|^{\alpha k},
$$
which is impossible.  Thus the second assertion follows.  
\end{proof}

\section{Interior $C^{1,1}-$regularity of $K_p(z)$}\label{sec:interior_on}

We first verify the following lemma, which also holds for $p=1$.

\begin{lemma}\label{lm:C1/2_K_1}
For $1\le p<\infty$ and $j=1,\cdots,2n$,
\[
z\mapsto \left.\frac{\partial{K_p(\cdot,z)}}{\partial{x_j}}\right|_z\in{\mathrm{Lip}_{\rm loc}(\Omega)}.
\]
The same conclusions also hold for $m_p(\cdot,z)$.
\end{lemma}

\begin{proof}
Let $U\subset\subset\Omega$ be a neighbourhood of $z$ which omit $S_z$. By Theorem \ref{th:Holder_off_diagonal} and \ref{th:Holder_off_diagonal_K_1}, for any $z'\in U\setminus \{z\}$,  the function
\[
h:=\frac{K_p(\cdot,z)-K_p(\cdot,z')}{z-z'}
\]
is  holomorphic  on $\Omega$ and satisfies $\sup_U|{h}|\leq{C}<\infty$ for any $1\le p<\infty$.   By Cauchy's estimates,  we conclude that
\begin{equation}\label{eq:derivative_h}
\left|\frac{\partial{h}}{\partial{x_j}}(z')\right|\leq{MC},\ \ \ z,z'\in{U'}\subset\subset{U},
\end{equation}
where the constant $M$ depends only on $U'$, $U$ and $\Omega$.  

On the other hand,  since $\partial{K_p(\cdot,z)}/\partial{x_j}$ is  holomorphic  on $\Omega$,  it follows that
\[
\Bigg|\left.\frac{\partial{K_p(\cdot,z)}}{\partial{x_j}}\right|_z-\left.\frac{\partial{K_p(\cdot,z)}}{\partial{x_j}}\right|_{z'}\Bigg| \leq {C'|z-z'|},
\]
where $C'=C'(U',\Omega)>0$.  Hence
\begin{eqnarray*}
\Bigg|\frac{\partial{K_p(\cdot,z)}}{\partial{x_j}}\bigg|_z-\frac{\partial{K_p(\cdot,z')}}{\partial{x_j}}\bigg|_{z'}\Bigg| &\leq& \Bigg|\frac{\partial{K_p(\cdot,z)}}{\partial{x_j}}\bigg|_z-\frac{\partial{K_p(\cdot,z)}}{\partial{x_j}}\bigg|_{z'}\Bigg|+\left|\frac{\partial{h}}{\partial{x_j}}(z')\right|\cdot|z-z'|\\
&\leq&{C'|z-z'|+MC|z-z'|}
\end{eqnarray*}
whenever $z'$ is sufficiently close to  $z$.
\end{proof}

\begin{proof}[Proof of Theorem \ref{th:Main}]
Suppose that $1\leq p<\infty$. Let $e_1,\cdots,e_{2n}$ be the standard basis in $\mathbb{R}^{2n}=\mathbb{C}^n$ and $t\in\mathbb{R}$. Use \eqref{eq:ineq_geq1_1}, we obtain
\begin{eqnarray*}
& & K_p(z+te_j)^{-1}-K_p(z)^{-1}\\
&=& m_p(z+te_j)^p-m_p(z)^p\\
&=& \|m_p(\cdot,z+te_j)\|^p_p-\|m_p(\cdot,z)\|^p_p\\
&\leq& p\,\mathrm{Re}\int_\Omega|m_p(\cdot,z+te_j)|^{p-2}\overline{m_p(\cdot,z+te_j)}(m_p(\cdot,z+te_j)-m_p(\cdot,z))
\end{eqnarray*}
This together with \eqref{eq:RPF} and the fact that $m_p(z+te_j,z+te_j)=m_p(z,z)=1$ give
\begin{eqnarray}\label{eq:difference I}
K_p(z+te_j)^{-1}-K_p(z)^{-1} &\leq& \frac{p}{K_p(z+te_j)}\mathrm{Re}\left\{m_p(z+te_j,z+te_j)-m_p(z+te_j,z)\right\}\nonumber\\
&=& -\frac{p}{K_p(z+te_j)}\mathrm{Re}\left\{m_p(z+te_j,z)-m_p(z,z)\right\}\nonumber\\
&=& -\frac{pt}{K_p(z+te_j)}\,\mathrm{Re}\left.\frac{\partial{m_p(\cdot,z)}}{\partial{x_j}}\right|_z+O(|t|^2)\nonumber\\
&=& -\frac{pt}{K_p(z)}\,\mathrm{Re}\left.\frac{\partial{m_p(\cdot,z)}}{\partial{x_j}}\right|_z+O(|t|^2),
\end{eqnarray}
in view of the local Lipschitz continuity of $K_p(\cdot)$.

Analogously,  
\begin{eqnarray}\label{eq:difference II}
K_p(z+te_j)^{-1}-K_p(z)^{-1}
&\geq& p\,\mathrm{Re}\int_\Omega|m_p(\cdot,z)|^{p-2}\overline{m_p(\cdot,z)}(m_p(\cdot,z+te_j)-m_p(\cdot,z))\nonumber\\
&=& -\frac{pt}{K_p(z)}\,\mathrm{Re}\left.\frac{\partial{m_p(\cdot,z+te_j)}}{\partial{x_j}}\right|_z+O(|t|^2).
\end{eqnarray}
Since
\[
\left.\frac{\partial{m_p(\cdot,z+te_j)}}{\partial{x_j}}\right|_{z}=\left.\frac{\partial{m_p(\cdot,z)}}{\partial{x_j}}\right|_{z}+O(|t|)
\]
in view of Lemma \ref{lm:C1/2_K_1}, it follows from \eqref{eq:difference II} that
\[
K_p(z+te_j)^{-1}-K_p(z)^{-1} \geq -\frac{pt}{K_p(z)}\,\mathrm{Re}\left.\frac{\partial{m_p(\cdot,z)}}{\partial{x_j}}\right|_{z}+O(|t|^2).
\]
This together with \eqref{eq:difference I} give
\[
\frac{\partial{K_p^{-1}}}{\partial{x_j}}(z)=-\frac{p}{K_p(z)}\,\mathrm{Re}\left.\frac{\partial{m_p(\cdot,z)}}{\partial{x_j}}\right|_{z},
\]
which implies
\[
\frac{\partial{K_p}}{\partial{x_j}}(z)=-K_p(z)^2\frac{\partial{K_p^{-1}}}{\partial{x_j}}(z)=p\,\mathrm{Re}\frac{\partial{K_p(\cdot,z)}}{\partial{x}_j}\bigg|_z,
\]
i.e., the first formula in \eqref{eq:partial derivative} holds. This combined with Lemma \ref{lm:C1/2_K_1} yields  $K_p(\cdot)\in{C^{1,1}_{\rm loc}(\Omega)}$ for any $1\leq p<\infty$.

It remains to verify the second formula in \eqref{eq:partial derivative}. By definition, we have
\begin{eqnarray*}
\frac{\partial{K_p}}{\partial{z_j}}(z) &=& \frac{1}{2}\left(\frac{\partial{K_p}}{\partial{x_j}}(z)-i\,\frac{\partial{K_p}}{\partial{x_{n+j}}}(z)\right)\\
 &=& \frac{p}{2}\left(\mathrm{Re}\frac{\partial{K_p}(\cdot,z)}{\partial{x_j}}\bigg|_z-i\,\mathrm{Re}\frac{\partial{K_p}(\cdot,z)}{\partial{x_{n+j}}}\bigg|_z\right).
\end{eqnarray*}
Since $K_p(\cdot,z)$ is holomorphic for fixed $z$,  we have 
\[
\frac{\partial{K_p}(\cdot,z)}{\partial{x_j}}=\frac{\partial{K_p}(\cdot,z)}{\partial{z_j}},\ \ \ \frac{\partial{K_p}(\cdot,z)}{\partial{x_{n+j}}}=i\,\frac{\partial{K_p}(\cdot,z)}{\partial{z_j}},
\]
so that
\begin{eqnarray*}
\frac{\partial{K_p}}{\partial{z_j}}(z) &=& \frac{p}{2}\left(\mathrm{Re}\frac{\partial{K_p}(\cdot,z)}{\partial{z_j}}\bigg|_z-i\,\mathrm{Re}\left\{i\,\frac{\partial{K_p}(\cdot,z)}{\partial{z_j}}\bigg|_z\right\}\right)\\
&=& \frac{p}{2}\left(\mathrm{Re}\frac{\partial{K_p}(\cdot,z)}{\partial{z_j}}\bigg|_z+i\,\mathrm{Im}\frac{\partial{K_p}(\cdot,z)}{\partial{z_j}}\bigg|_z\right)\\
&=& \frac{p}{2}\cdot \frac{\partial{K_p}(\cdot,z)}{\partial{z_j}}\bigg|_z.
\end{eqnarray*}
\end{proof}

As is indicated in \eqref{eq:partial derivative} and the proof of Lemma \ref{lm:C1/2_K_1}, one can actually show that $K_p(\cdot)$ is locally $C^{k+1,\alpha}$ provided that $K_p(\zeta,\cdot)$ is locally $C^{k,\alpha}$. It seems to be a big challenge to get the differentiablity of $K_p(\zeta,\cdot)$. A simple partial result is Corollary \ref{cor:differentiable_z_z}.

\begin{proof}[Proof of Corollary \ref{cor:differentiable_z_z}]
Let $\xi=(\xi_1,\cdots,\xi_n)\in\mathbb{C}^n$. By Theorem \ref{th:Main}, we have
\begin{eqnarray*}
K_p(z,z+\xi)-K_p(z,z)
&=& K_p(z,z+\xi)-K_p(z+\xi,z+\xi)+K_p(z+\xi)-K_p(z)\\
&=& -\sum^n_{j=1}\left.\frac{\partial{K_p(\cdot,z+\xi)}}{\partial{z}_j}\right|_z\xi_j+2\mathrm{Re}\sum^n_{j=1}\frac{\partial{K_p(z)}}{\partial{z}_j}\xi_j+o(|\xi|).
\end{eqnarray*}
Here, the coefficient in the error term $o(|\xi|)$ depends on $D^2K_p(\cdot,z+\xi)|_z$, which is uniformly bounded as $\xi\rightarrow0$. Since
\[
\frac{K_p(\cdot,z+\xi)-K_p(\cdot,z)}{\xi}
\]
is locally bounded, we may use Cauchy's estimates in a similar way as the proof of Lemma \ref{lm:C1/2_K_1} to derive
\[
\Bigg|\left.\frac{\partial{K_p(\cdot,z+\xi)}}{\partial{z}_j}\right|_z-\left.\frac{\partial{K_p(\cdot,z)}}{\partial{z}_j}\right|_z\Bigg|=O(|\xi|).
\]
Thus
\[
K_p(z,z+\xi)-K_p(z,z)=-\sum^n_{j=1}\left.\frac{\partial{K_p(\cdot,z)}}{\partial{z}_j}\right|_z\xi_j+2\mathrm{Re}\sum^n_{j=1}\frac{\partial{K_p(z)}}{\partial{z}_j}\xi_j+o(|\xi|),
\]
i.e., $K_p(z,\cdot)$ is differentiable at $z$.
\end{proof}

The argument above and \eqref{eq:partial derivative} also yield
\begin{equation}\label{eq:partial_derivative_second}
\mathrm{Re}\left.\frac{\partial{K_p(z,\cdot)}}{\partial{x_j}}\right|_z=(p-1)\mathrm{Re}\left.\frac{\partial{K_p(\cdot,z)}}{\partial{x_j}}\right|_z,\ \ \ \mathrm{Im}\left.\frac{\partial{K_p(z,\cdot)}}{\partial{x_j}}\right|_z=-\mathrm{Im}\left.\frac{\partial{K_p(\cdot,z)}}{\partial{x_j}}\right|_z.
\end{equation}

\begin{proof}[Proof of Corollary \ref{cor:Image}]
When $p=1$, we have
\[
\mathrm{Re}\left.\frac{\partial{K_1(z,\cdot)}}{\partial{x_j}}\right|_z=0,
\]
in view of \eqref{eq:partial_derivative_second}. Thus the image of the function
\[
z\mapsto\left.\frac{\partial{K_1(z,\cdot)}}{\partial{x_j}}\right|_z
\]
lies on the imaginary axis. 
\end{proof}

The case when $p>1$ is different.  For example,  if $\Omega=\mathbb{D}$ is the unit disc in $\mathbb{C}$, then
\[
K_p(z,w)=\frac{1}{\pi}\frac{(1-|z|^2)^{4/p-2}}{(1-\zeta\overline{z})^{4/p}}
\]
(see \cite{CZ},   Proposition 2.9),  so that
\begin{eqnarray*}
\left.\frac{\partial{K_p(z,\cdot)}}{\partial{x}}\right|_z
&=& \left.\frac{\partial{K_p(z,w)}}{\partial{w}}\right|_{w=z}+\left.\frac{\partial{K_p(z,w)}}{\partial\overline{w}}\right|_{w=z}\\
&=& \frac{2}{\pi(1-|z|^2)^3}\left(z-\frac{2-p}{p}\overline{z}\right),\ \ \ z\in\mathbb{D}
\end{eqnarray*}
where $x=\mathrm{Re}\,w$. It is easy to verify that the image of such a function is the whole complex plane when $p>1$.

\section{Parameter dependence}\label{sec:parameter}
The materials of this section are taken from the unpublished preprint \cite{CX}. Consider the following weighted $L^2$ Bergman space:
$$
A^2(\Omega,\varphi):=\left\{ f\in \mathcal O(\Omega): \int_\Omega |f|^2 e^{-\varphi}<\infty\right\}.
$$
According to Corollary 3.1 in \cite{PW},  $A^2(\Omega,\varphi)$ admits a unique Bergman reproducing kernel $K_{\Omega,\varphi}(\cdot,\cdot)$ if for each $z\in \Omega$ there exists a number $\alpha>0$ such that $e^{\alpha\varphi}$ is integrable in a neighborhood of $z$. In particular, we have
\begin{equation}\label{eq:def_weighted_kernel}
K_{\Omega,\varphi}(z)=K_{\Omega,\varphi}(z,z)=\sup\left\{|f(z)|^2 :f\in A^2(\Omega,\varphi), \int_\Omega |f|^2 e^{-\varphi}=1\right\}.
\end{equation}
Given $p\ge 1$,  define
$$
A^2_{p,z}(\Omega):=A^2(\Omega,(2-p)\log |m_p(\cdot,z)|)\ \ \text{and}\ \ K_{2,p,z}(\cdot,\cdot):=K_{\Omega,(2-p)\log |m_p(\cdot,z)|}(\cdot,\cdot).  
$$ 
 
\begin{theorem}\label{th:p<2}
Let $1\le p\le 2$.  Then we have
\begin{equation}\label{eq:RF_0}
K_p(\cdot,z)=K_{2,p,z}(\cdot,z),\ \ \ \forall\,z\in \Omega.
\end{equation}
\end{theorem}

\begin{proof}
We first verify that $A^2_{p,z}(\Omega)\subset A^p(\Omega)$ for $1\le p\le 2$.  To see this, simply note that for any $f\in A^2_{p,z}(\Omega)$,
\begin{eqnarray*}
\int_\Omega |f|^p & = & \int_\Omega \left(|m_p(\cdot,z)|^{p(p-2)/2}|f|^p\right)|m_p(\cdot,z)|^{p(2-p)/2}\\
& \le & \left(\int_\Omega |m_p(\cdot,z)|^{p-2} |f|^2\right)^{p/2} \left(\int_\Omega |m_p(\cdot,z)|^p \right)^{1-p/2}\\
& = & m_p(z)^{p(1-p/2)} \left(\int_\Omega |m_p(\cdot,z)|^{p-2} |f|^2\right)^{p/2}\\
& < & \infty.
\end{eqnarray*}
Thus  we have
\begin{equation}\label{eq:RF_1}
f(z)= \int_\Omega |m_p(\cdot,z)|^{p-2}\overline{K_p(\cdot,z)}f,\ \ \ \forall\,f\in A^2_{p,z}(\Omega),
\end{equation}
in view of \eqref{eq:RPF}.
On the other hand,  the reproducing formula for $A^2_{p,z}(\Omega)$ gives
\begin{equation}\label{eq:RF_2}
f(z)= \int_\Omega |m_p(\cdot,z)|^{p-2}\overline{K_{2,p,z}(\cdot,z)}f,\ \ \ \forall\,f\in A^2_{p,z}(\Omega).
\end{equation}
Thus
\begin{equation}\label{eq:RF_3}
 \int_\Omega |m_p(\cdot,z)|^{p-2}\overline{(K_p(\cdot,z)-K_{2,p,z}(\cdot,z))}f=0,\ \ \ \forall\,f\in A^2_{p,z}(\Omega).
\end{equation}
Since 
\begin{equation}\label{eq:RF_4}
 \int_\Omega |m_p(\cdot,z)|^{p-2}|K_p(\cdot,z)|^2=K_p(z)^2 \int_\Omega |m_p(\cdot,z)|^p=
 K_p(z),  
\end{equation}
we see that $K_p(\cdot,z)\in A^2_{p,z}(\Omega)$.   Substitute $f:=K_p(\cdot,z)-K_{2,p,z}(\cdot,z)$ into \eqref{eq:RF_3},  we immediately get $f=0$,  i.e.,  \eqref{eq:RF_0} holds.
\end{proof}

Recall that the hyperconvexity index $\alpha(\Omega)$ is given by
\[
\alpha(\Omega):=\sup\left\{\alpha\geq0;\ \exists\,\rho\in{PSH^-(\Omega)}\cap{C(\Omega)},\ \text{s.t.}\ -\rho\lesssim \delta_\Omega^\alpha\right\}.
\]
As an application of Theorem \ref{th:p<2},  we have

\begin{proposition}\label{prop:Integ}
Let $1\leq{p}\leq2$. If\/ $\alpha(\Omega)>0$,  then
$K_p(\cdot,z)\in L^q(\Omega)$\/ for any $z\in \Omega$ and $q< 2pn/(2n-\alpha(\Omega))$.
\end{proposition}

Let $A^2(\Omega,\varphi)$ and  $K_{\Omega,\varphi}$ be given as \S\,1.  We need the following $L^2$ boundary decay estimate for $K_{\Omega,\varphi}(\cdot,z)$.

\begin{proposition}\label{prop:BergmanIntegral}
Let $\Omega\subset {\mathbb C}^n$ be a pseudoconvex domain and $\rho$ a negative continuous psh  function on $\Omega$.
 Set
  $$
 \Omega_t=\{\zeta\in \Omega:-\rho(\zeta)>t\},\ \ \ t>0.
 $$
   Let $a>0$ be given.
  For every $0<r<1$, there exist constants $\varepsilon_r,C_r>0$ such that
  \begin{equation}\label{eq:2.1}
  \int_{-\rho\le \varepsilon} |K_{\Omega,\varphi}(\cdot,z)|^2e^{-\varphi} \le C_{r}\, K_{\Omega_a,\varphi}(z) (\varepsilon/a)^{r}
  \end{equation}
  for all $z\in \Omega_a$ and $\varepsilon\le \varepsilon_r a$.
  \end{proposition} 
  
  Proposition \ref{prop:BergmanIntegral} has been verified in \cite{Chen17} for the case $\varphi=0$.  Although the general case is  the same,   we still include a proof here for the sake of completeness.
  
  \begin{proof}
  Let $P_{\Omega,\varphi}: L^2(\Omega,\varphi)\rightarrow A^2(\Omega,\varphi)$ be the Bergman projection.   Then we have
  \begin{equation}\label{eq:BCh}
  \int_\Omega |P_{\Omega,\varphi}(f)|^2 (-\rho)^{-r}e^{-\varphi} \le C_r \int_\Omega |f|^2 (-\rho)^{-r}e^{-\varphi}
  \end{equation}
  for any $0<r<1$ and $f\in L^2(\Omega,\varphi)$ (cf. \cite{BCh}).  Let $\chi_E$ denote the characteristic function of a set $E$.  
  For $f:=\chi_{\Omega_a} K_{\Omega_a,\varphi}(\cdot,z)$,  we have  
\begin{eqnarray}\label{eq:BCh_2}
P_{\Omega,\varphi}(f)(\zeta) & = & \int_\Omega \chi_{\Omega_a}(\cdot) K_{\Omega_a,\varphi}(\cdot,z)\overline{K_{\Omega,\varphi}(\cdot,\zeta)}\nonumber\\
& = & \overline{\int_{\Omega_a} K_{\Omega,\varphi}(\cdot,\zeta) \overline{K_{\Omega_a,\varphi}(\cdot,z)}}\nonumber\\
& = &  K_{\Omega,\varphi}(\zeta,z)
\end{eqnarray} 
in view of the reproducing property.  
   By \eqref{eq:BCh} and \eqref{eq:BCh_2},  we obtain
    \begin{eqnarray*}
   \int_\Omega |K_{\Omega,\varphi}(\cdot,z)|^2 (-\rho)^{-r} e^{-\varphi} & \le & C_r \int_{\Omega_a} | K_{\Omega_a,\varphi}(\cdot,z) |^2 (-\rho)^{-r} e^{-\varphi}\\
   & \le & C_r a^{-r} K_{\Omega_a,\varphi}(z),
    \end{eqnarray*}
    from which the assertion immediately follows,  for 
    $$
    \int_\Omega |K_{\Omega,\varphi}(\cdot,z)|^2 (-\rho)^{-r} e^{-\varphi} \ge \varepsilon^{-r}\int_{-\rho\le \varepsilon} |K_{\Omega,\varphi}(\cdot,z)|^2 e^{-\varphi}.
    $$
  \end{proof}
  
  We also need the following Bergman type inequality.

\begin{lemma}\label{lm:Bergman_ineq}
Suppose that $1\leq{p}\leq2$. Let $a>0$ and $d(z):=d(z,\partial\Omega_a)$. For $\varphi:=(2-p)\log|m_p(\cdot,z)|$, we have
$$
K_{\Omega_{a},\varphi}(z)\le C_n^{2/p} d(z)^{-4n/p} |\Omega|^{2/p-1},
$$  
where $C_n$ denotes a generic constant which depends on $n$.
\end{lemma}

\begin{proof}

For any $f\in{A^2(\Omega_a,\varphi)}\setminus\{0\}$, we have
\begin{eqnarray*}
|f(z)|^2
&\leq& \frac{1}{|B(z,d(z)/2)|}\int_{B(z,d(z)/2)}|f|^2\\
&\leq& {C_n}d(z)^{-2n}\sup_{B(z,d(z)/2)}{e^\varphi}\int_{B(z,d(z)/2)}|f|^2e^{-\varphi}\\
&\leq& {C_n}d(z)^{-2n}\sup_{B(z,d(z)/2)}{e^\varphi}\int_{\Omega_a}|f|^2e^{-\varphi},
\end{eqnarray*}
in view of the mean-value inequality. It follows from \eqref{eq:def_weighted_kernel} that
\begin{equation}\label{eq:Bergman_ineq_1}
K_{\Omega_a,\varphi}(z)\leq{C_n}d(z)^{-2n}\sup_{B(z,d(z)/2)}{e^\varphi}=C_n d(z)^{-2n} \left(\sup_{B(z,d(z)/2)} |m_p(\cdot,z)|^p\right)^{2/p-1}.
\end{equation}
For any $\zeta\in{B(z,d(z)/2)}$, we have $B(\zeta,d(z)/2)\subset\Omega_a\subset\Omega$. Again, since $|m_p(\cdot,z)|^p$ is a plurisubharmonic function, the mean-value inequality yields that
\begin{eqnarray*}
|m_p(\zeta,z)|^p
&\leq& \frac{1}{|B(\zeta,d(z)/2)|}\int_{B(\zeta,d(z)/2)}|m_p(\cdot,z)|^p\\
&\leq& \frac{1}{|B(\zeta,d(z)/2)|}\int_\Omega|m_p(\cdot,z)|^p\\
&=& C_n d(z)^{-2n} m_p(z)^p\\
&=& \frac{C_nd(z)^{-2n}}{K_p(z)}.
\end{eqnarray*}
Since
\[
K_p(z)=\sup_{f\in{A^p(\Omega)}\setminus\{0\}}\frac{|f(z)|^p}{\|f\|_p^p}\geq\frac{1}{|\Omega|},
\]
it follows that
\[
|m_p(\zeta,z)|^p\leq{C_nd(z)^{-2n}|\Omega|},\ \ \ \forall\,\zeta\in{B(z,d(z)/2)}.
\]
This together with \eqref{eq:Bergman_ineq_1} yield the assertion.
\end{proof}
  
  \begin{proof}[Proof of Proposition \ref{prop:Integ}]
 For any $0<\alpha<\alpha(\Omega)$,  we take a continuous negative plurisubharmonic function $\rho$ on $\Omega$ such that $-\rho\le C_\alpha \delta^\alpha$ for suitable constant $C_\alpha>0$.  By Theorem \ref{th:p<2},  we have $K_{2,p,z}(\cdot,z)=K_p(z)m_p(\cdot,z)$.   Apply Proposition \ref{prop:BergmanIntegral} with   $\varphi:=(2-p)\log |m_p(\cdot,z)|$ and $a=-\rho(z)/2$,  we obtain for $\varepsilon\ll1$,
  \begin{eqnarray*}
  \int_{-\rho\le \varepsilon} |m_p(\cdot,z)|^p  & = & K_p(z)^{-2} \int_{-\rho\le \varepsilon} |m_p(\cdot,z)|^{p-2} |K_{2,p,z}(\cdot,z)|^2\\
  & \le & C_{r} \frac{K_{\Omega_{a},\varphi}(z)}{K_p(z)^2}\left(\frac{\varepsilon}{-\rho(z)}\right)^r.
  \end{eqnarray*}
Define $d(z)=d(z,\partial \Omega_a)$. As $1\le p\le 2$, it follows from Lemma \ref{lm:Bergman_ineq} and the fact $K_p(z)\geq1/|\Omega|$ that
  \begin{equation}\label{eq:B-Integ_1}
  \int_{\delta\le \varepsilon} |m_p(\cdot,z)|^p\lesssim \varepsilon^{r\alpha},
  \end{equation}
  where the implicit constant depends only on $n,r,\alpha,z$ and $|\Omega|$.  
  This combined with the mean-value inequality gives
  $$
  |m_p(\zeta,z)|^p \le C_n \delta(\zeta)^{-2n} \int_{\delta\le 2\delta(\zeta)} |m_p(\cdot,z)|^p\lesssim \delta(\zeta)^{r\alpha-2n},\ \ \ \forall\, \zeta\in\Omega.
  $$
  Let $\tau>0$.  Then we have
  \begin{eqnarray*}
  \int_\Omega |m_p(\cdot,z)|^{p+\tau} & \lesssim & 1+ \sum_{k=1}^\infty \int_{2^{-k-1}<\delta \le 2^{-k}} |m_p(\cdot,z)|^{p+\tau}\\
  & \lesssim & 1+ \sum_{k=1}^\infty 2^{k\tau(2n-r\alpha)/p} \int_{\delta \le 2^{-k}} |m_p(\cdot,z)|^{p}\\
  & \lesssim & 1+ \sum_{k=1}^\infty 2^{k\tau(2n-r\alpha)/p-kr\alpha}\\
  & < & \infty
  \end{eqnarray*}
  provided $\tau(2n-r\alpha)/p<r\alpha$,  i.e.,  $\tau<pr\alpha/(2n-r\alpha)$.   Since $r$ and $\alpha$ can be arbitrarily close to $1$ and $\alpha(\Omega)$ respectively,  we are done. 
  \end{proof} 
  
  \begin{proof}[Proof of Theorem \ref{th:LlogL}]
Without loss of generality,  we assume $s>p$.  Since $|\Omega|^{1/t}\cdot K_t(z)^{1/t}$ is  nonincreasing in $t$ (cf.  (6.3) in \cite{CZ}),  we see that 
  $$
  K_p(z)\ge |\Omega|^{p/s-1}\cdot K_s(z)^{p/s} = K_s(z)\left( |\Omega|K_s(z)\right)^{p/s-1} \ge K_s(z)- C(s-p)
  $$
  when $s$ is sufficiently close to $p$,
  where $C$ is  a suitable constant depending only on $z,\Omega$. On the other hand,  
 we infer from \eqref{eq:B-Integ_1}  that
\begin{eqnarray*}
\int_{\delta\le \varepsilon} |K_p(\cdot,z)|^{s} & \le &  \sum_{k=0}^\infty \int_{2^{-k-1}\varepsilon <\delta \le 2^{-k}\varepsilon} |K_p(\cdot,z)|^{p+s-p}\\
  & \lesssim &  \sum_{k=0}^\infty (2^{-k}\varepsilon)^{-(s-p)(2n-r\alpha)/p} \int_{\delta \le 2^{-k}\varepsilon} |K_p(\cdot,z)|^{p}\\
  & \lesssim &  \sum_{k=0}^\infty (2^{-k}\varepsilon)^{-(s-p)(2n-r\alpha)/p+r\alpha}\\
  & \lesssim & \varepsilon^{\alpha(\Omega)/2}
\end{eqnarray*} 
provided that $1-r$ and $s-p$ are sufficiently small.   
  Thus
  \begin{eqnarray*}
 \int_\Omega |K_p(\cdot,z)|^s & = &  \int_{\delta> \varepsilon} |K_p(\cdot,z)|^p |K_p(\cdot,z)|^{s-p} + \int_{\delta\le \varepsilon} |K_p(\cdot,z)|^s\\
 & \le & (C_1\varepsilon^{-2n/p})^{s-p} \int_\Omega |K_p(\cdot,z)|^p + C_2 \varepsilon^{\alpha(\Omega)/2}.
  \end{eqnarray*}
  Now we take $\varepsilon=(s-p)^{2/\alpha(\Omega)}$.  Since
  $$
  (C_1\varepsilon^{-2n/p})^{s-p}\le 1+ C_3 (s-p)|\log(s-p)|\ \ \ \text{and}\ \ \ \varepsilon^{\alpha(\Omega)/2}=s-p,
  $$
  it follows that
  \begin{eqnarray*}
  \int_\Omega |K_p(\cdot,z)|^s & \le & \int_\Omega |K_p(\cdot,z)|^p + C_4 (s-p)|\log (s-p)|\\
  & = & K_p(z)^{p-1} + C_4 (s-p)|\log (s-p)|.
  \end{eqnarray*}
  Thus 
  \begin{eqnarray*}
  K_s(z)\ge \frac{K_p(z)^s}{\int_\Omega |K_p(\cdot,z)|^s} & \ge & \frac{K_p(z)^p -C_5(s-p)}{K_p(z)^{p-1}+C_4(s-p)|\log(s-p)|}\\
  &\ge&  K_p(z)-C_6(s-p)|\log(s-p)|.
  \end{eqnarray*}
  \end{proof}

Finally, we mention that both Theorem \ref{th:LlogL} and Theorem \ref{th:p<2} fail  for  $2<p<\infty$ in general (see \cite{CX}, Section 4). It is interesting to ask the following question:

\begin{problem}[B{\l}ocki]
Does there exist for each $2<p<\infty$ and $z\in \Omega$ a weight function $\varphi_z$ such that $K_p(\cdot,z)=K_{\Omega,\varphi_z}(\cdot,z)$?
\end{problem}

\section{An upper estimate for $i\partial\overline{\partial}\log{K_p(z;X)}$ ($1<p<2$)}

\begin{proof}[Proof of Theorem \ref{th:Levi_upper_p_le_2}]
Let $1<p<2$ and $1/p+1/q=1$, so that $q>2$. We shall apply the duality method. Set
\begin{equation}\label{eq:gz}
g_z:=\frac{|m_p(\cdot,z)|^{p-2}m_p(\cdot,z)}{m_p(z)^p}.
\end{equation}
Since $\|m_p(\cdot,z)\|_p=m_p(z)$,  we have $g_z\in{L^q(\Omega)}$ and 
\begin{equation}\label{eq:gz_norm}
\lVert{g_z}\rVert_q=m_p(z)^{-1}=K_p(z)^{1/p}=\lVert{L_z}\rVert.
\end{equation}
Moreover, it follows from \eqref{eq:RPF} that
\begin{equation}\label{eq:reproducing I}
L_z(f)=f(z)=\int_\Omega{f}\overline{g}_z,\ \ \ \forall\, f\in{A^p(\Omega)}.
\end{equation}
Our goal is to find an effective upper bound of
\[
K_p(z+tv)^{q/p}-K_p(z)^{q/p} = \lVert{L_{z+tv}}\rVert^q-\lVert{L_z}\rVert^q = \|g_{z+tv}\|^q_q-\|g_z\|^q_q
\]
for any unit vector $v$ in $\mathbb{R}^{2n}=\mathbb{C}^n$, in view of Lemma \ref{lm:Levi_upper_estimate}.

We introduce the following linear functional
\[
L_{z,v}(f):=\frac{\partial{f}}{\partial{v}}(z),\ \ \ f\in{A^p(\Omega)}.
\]
Cauchy's estimates imply that  $L_{z,v}$ is a bounded linear functional. By the Hahn-Banach theorem, $L_{z,v}$ can be extended to a continuous linear functional $\widetilde{L}_{z,v}\in{(L^p(\Omega))^*}$ with $\lVert{L_{z,v}}\rVert=\lVert{\widetilde{L}_{z,v}}\rVert$.  Let $g_{z,v}\in{L^q(\Omega)}$  be the unique representation of $\widetilde{L}_{z,v}$, so that $\lVert{g_{z,v}}\rVert_q=\lVert{\widetilde{L}_{z,v}}\rVert$ and
\begin{equation}\label{eq:reproducing II}
L_{z,v}(f)=\widetilde{L}_{z,v}(f)=\int_\Omega{f}\overline{g}_{z,v},\ \ \ \forall\, f\in{A^p(\Omega)}.
\end{equation}
\eqref{eq:reproducing I} together with \eqref{eq:reproducing II} yield
\[
(L_z+tL_{z,v})(f)=\int_\Omega{f}\,\overline{g_z+tg_{z,v}},\ \ \ \forall\,{f\in{A^p(\Omega)}}.
\]
In what follows we assume that $\Omega'\subset\subset\Omega''\subset\subset\Omega$ are open sets such that $z\in\Omega'$, $\varepsilon_0<\min\{1,d(\Omega',\partial\Omega'')\}$ and $0<t\ll\varepsilon_0$. Apply \eqref{eq:ineq_geq1_1} for $a=g_{z+tv}$ and $b=g_z+tg_{z,v}$, we get
\begin{eqnarray}\label{eq:regular_upper}
K_p(z+tv)^{q/p}-K_p(z)^{q/p}
&=& \|g_{z+tv}\|_q^q-\|g_z\|_q^q\nonumber\\
&\leq& \|g_z+tg_{z,v}\|_q^q-\|g_z\|_q^q\nonumber\\
& & + q\,\mathrm{Re}\int_\Omega|g_{z+tv}|^{q-2}\overline{g}_{z+tv}(g_{z+tv}-g_z-tg_{z,v})\nonumber\\
&=:& \|g_z+tg_{z,v}\|_q^q-\|g_z\|_q^q+I.
\end{eqnarray}
A direct computation yields
\[
|g_z|^{q-2}\overline{g}_z=\frac{|m_p(\cdot,z)|^{(p-1)(q-2)+(p-2)}\overline{m_p(\cdot,z)}}{m_p(z)^{p(q-1)}}=\frac{\overline{m_p(\cdot,z)}}{m_p(z)^q},
\]
for $(p-1)(q-2)+p-2=0$ and $p(q-1)=q$. This together with \eqref{eq:reproducing I} and \eqref{eq:reproducing II} imply that
\begin{eqnarray*}
I &=& \frac{q}{m_p(z+tv)^q}\,\mathrm{Re}\int_\Omega\overline{m_p(\cdot,z+tv)}(g_{z+tv}-g_z-tg_{z,v})\nonumber\\
&=& \frac{q}{m_p(z+tv)^q}\,\mathrm{Re}\big\{(L_{z+tv}-L_z-tL_{z,v})(m_p(\cdot,z+tv))\big\}\nonumber\\
&=& \frac{q}{m_p(z+tv)^q}\,\mathrm{Re}\left\{m_p(z+tv,z+tv)-m_p(z,z+tv)-t\left.\frac{\partial{m_p(\cdot,z+tv)}}{\partial{v}}\right|_z\right\}\nonumber\\
&\le & \frac{q}{m_p(z+tv)^q}\,\mathrm{Re}\left\{\frac{t^2}{2}\cdot\left.\frac{\partial^2{m_p(\cdot,z+tv)}}{\partial{v^2}}\right|_z\right\}+\frac{1}{6}\sup_{\Omega''}|D^3m_p(\cdot,z+tv)|\cdot|t|^3\\
&\le & \frac{q}{m_p(z)^q}\,\mathrm{Re}\left\{\frac{t^2}{2}\cdot\left.\frac{\partial^2{m_p(\cdot,z)}}{\partial{v^2}}\right|_z\right\}+C|t|^3
\end{eqnarray*}
for some generic constant $C=C(\Omega'',\Omega)>0$. Here we have used the the fact that  
\[
z\mapsto\left.\frac{\partial^2{m_p(\cdot,z)}}{\partial{v^2}}\right|_z
\]
is locally Lipschitz continuous,  whose  proof is essentially the same as that of Lemma \ref{lm:C1/2_K_1}.

Recall that one may identify $v$ with certain complex tangent vector  $X$,   and $e^{i\theta}v$  is the real vector corresponding to $e^{i\theta}X$.   We claim that
\begin{equation}\label{eq:g_theta}
g_{z,e^{i\theta}v}=e^{-i\theta}g_{z,v}.
\end{equation}
To see this,   note that
\[
L_{z,e^{i\theta}v}(f)=e^{i\theta}Xf(z)=e^{i\theta}\frac{\partial{f}}{\partial{v}}(z)=e^{i\theta}L_{z,v}(f)=\int_\Omega f\overline{e^{-i\theta}g_{z,v}},\ \ \ \forall\,f\in{A^p(\Omega)}
\]
in view of \eqref{eq:complex_real_derivative_holomophic},  from which  \eqref{eq:g_theta} follows since 
 $$
 \|e^{-i\theta}g_{z,v}\|_q=\|g_{z,v}\|_q=\|L_{z,v}\|=\|L_{z,e^{i\theta}v}\|.
 $$
 Moreover,  \eqref{eq:complex_real_derivative_holomophic} also yields
\[
\frac{\partial^2{m_p(\cdot,z)}}{\partial{(e^{i\theta}v)^2}}=e^{2i\theta}X^2m_p(\cdot,z).
\]
Thus
\begin{eqnarray}\label{eq:Levi_upper_1}
&& K_p(z+te^{i\theta}X)^{q/p}-K_p(z)^{q/p}\\
 &\leq& \|g_z+te^{-i\theta}g_{z,v}\|_q^q-\|g_z\|_q^q+\frac{q}{m_p(z)^q}\mathrm{Re}\left\{\frac{t^2e^{2i\theta}}{2}\cdot\left.X^2m_p(\cdot,z)\right|_z\right\}+C|t|^3.\nonumber
\end{eqnarray}

Next, we apply \eqref{eq:ineq_geq4} to estimate the difference $\|g_z+te^{-i\theta}g_{z,v}\|_q^q-\|g_z\|_q^q$. It follows that
\begin{eqnarray}
& &\|g_z+te^{-i\theta}g_{z,v}\|_q^q-\|g_z\|_q^q\nonumber\\
&\leq& qt\,\mathrm{Re}\left\{e^{-i\theta}\int_\Omega|g_z|^{q-2}\overline{g}_zg_{z,v}\right\} + \frac{q^2t^2}{4}\|g_z\|_q^{q-2}\|g_{z,v}\|_q^2\nonumber\\
& & + \frac{q(q-2)t^2}{4}\,\mathrm{Re}\left\{e^{-2i\theta}\int_\Omega|g_z|^{q-4}\overline{g}_z^2g_{z,v}^2\right\}+R_q(t),\label{eq:Levi_upper_2}
\end{eqnarray}
where
\[
R_q(t) = o(t^2),\ \ \ (t\rightarrow0+)
\]
holds uniformly with respect to $\theta$, in view of H\"{o}lder's inequality.

Set
\[
\widehat{\mathcal{M}}_p^{(1)}(z;X):=\sup\left\{|Xf(z)|: f\in{A^p(\Omega)},\ \|f\|_p=1\right\}.
\]
Then we have $\|L_{z,v}\|=\widehat{\mathcal{M}}_p^{(1)}(z;X)$. By \eqref{eq:Levi_upper_1} and \eqref{eq:Levi_upper_2}, we have
\begin{eqnarray}\label{eq:Levi_upper_3}
&& K_p(z+te^{i\theta}X)^{q/p}-K_p(z)^{q/p}\nonumber\\
&\leq& qt\,\mathrm{Re}\left\{e^{-i\theta}\int_\Omega|g_z|^{q-2}\overline{g}_zg_{z,v}\right\} +\frac{q^2t^2}{4}K_p(z)^{(q-2)/p}\widehat{\mathcal{M}}_p^{(1)}(z;X)^2\nonumber\\
& & +\frac{q t^2}{2m_p(z)^q}\,\mathrm{Re}\left\{e^{2i\theta}\left.X^2m_p(\cdot,z)\right|_z\right\}\nonumber\\
& & + \frac{q(q-2)t^2}{4}\,\mathrm{Re}\left\{e^{-2i\theta}\int_\Omega|g_z|^{q-4}\overline{g}_z^2g_{z,v}^2\right\} + o(t^2).
\end{eqnarray}
Integrate with respect to $\theta$ from 0 to $2\pi$, we obtain
\begin{eqnarray*}
&& \frac{1}{t^2}\left(\frac{1}{2\pi}\int^{2\pi}_0K_p(z+te^{i\theta}X)^{q/p}d\theta-K_p(z)^{q/p}\right) \\
& \leq & \frac{q^2}{4} K_p(z)^{(q-2)/p}\widehat{\mathcal{M}}_p^{(1)}(z;X)^2+\varepsilon
\end{eqnarray*}
for any given $\varepsilon>0$,  provided $0<t\ll1$.  Letting $\varepsilon\rightarrow0$, we infer from Lemma \ref{lm:Levi_upper_estimate} that
\begin{equation}\label{eq:Levi_upper}
i\partial\bar{\partial} K_p^{q/p}(z;X)=X\overline{X}K_p(z)^{q/p}\leq\frac{q^2}{4}K_p(z)^{(q-2)/p}\widehat{\mathcal{M}}_p^{(1)}(z;X)^2
\end{equation}
Since
\begin{eqnarray*}
X\overline{X}\log{K_p}(z) &=& \frac{p}{q}X\overline{X}\log{K_p}(z)^{q/p}\\
&=& \frac{p}{q}\frac{X\overline{X}K_p(z)^{q/p}}{K_p(z)^{q/p}}-\frac{p}{q}\frac{|XK_p(z)^{q/p}|^2}{K_p(z)^{2q/p}}\\
&=& \frac{p}{q}\frac{X\overline{X}K_p(z)^{q/p}}{K_p(z)^{q/p}}-\frac{q}{p}\frac{|XK_p(z)|^2}{K_p(z)^2},
\end{eqnarray*}
we have
$$
i\partial\bar{\partial} \log{K_p}(z;X)=X\overline{X}\log{K_p}(z)\leq\frac{pq}{4}\frac{\widehat{\mathcal{M}}_p^{(1)}(z;X)^2}{K_p(z)^{2/p}}-\frac{q}{p}\frac{|XK_p(z)|^2}{K_p(z)^2}, 
$$
from which \eqref{eq:Levi_upper_0} immediately follows since $p=q/(q-1)$.
\end{proof}

\begin{remark}
Note that \eqref{eq:partial derivative} implies
\[
XK_p(z)=\frac{p}{2}XK_p(\cdot,z)|_z.
\]
It follows that
\begin{equation}\label{eq:B_p_K}
\widehat{B}_p(z;X)\geq{K_p(z)^{-1/p}}\frac{|XK_p(\cdot,z)|}{\|K_p(\cdot,z)\|_p}=\frac{|XK_p(\cdot,z)|}{K_p(z)}.
\end{equation}
Thus
\[
\frac{q}{p}\frac{|XK_p(z)|^2}{K_p(z)^2}\leq\frac{pq}{4}\widehat{B}_p(z;X)^2,
\]
which suggests that the coefficient $pq/4$ in \eqref{eq:Levi_upper_0} is probably the best. On the other hand, it was proved in \cite{CZ} that
\[
i\partial\bar{\partial} \log{K_p}(z;X)\geq{C(z;X)>0},
\]
where $C(z;X)$ is the Carath\'{e}odory metric. Hence \eqref{eq:B_p_K} is indeed a strict inequality.
\end{remark}

As a supplement of Theorem \ref{th:Levi_upper_p_le_2},  we present the following

\begin{proposition}\label{prop:Equality_p=2}
Equality in \eqref{eq:Levi_upper_0} holds for $p=q=2$.
\end{proposition}

\begin{proof}
Since 
$$
i\partial\bar{\partial} \log K_2(z;X)=\frac{X\overline{X}K_2(z)}{K_2(z)}-\frac{|X K_2(z)|^2}{K_2(z)^2},
$$
it suffices to show
\begin{equation}\label{eq:Levi_p=2}
X\overline{X}K_2(z)=\widehat{\mathcal{M}}_2^{(1)}(z;X)^2.
\end{equation}
We define
\[
A_0:=\{f\in{A^2(\Omega)}:\ f(z)=0\},\ \ \ A_1:=\{f\in A_0:\ Xf(z)=0\}.
\]
Take an orthonormal basis $\{h_0,h_1,h_2,\cdots\}$ of $A^2(\Omega)$ which satisfies $h_0\in{A_0^\perp}$,  $h_1\in{A_1^\perp\cap{A_0}}$ and $h_j\in A_1$ for $j\ge 2$.  Then we have $K_2(\cdot)=\sum^\infty_{j=0}|h_j(\cdot)|^2$,  so that 
\[
X\overline{X}K_2(z)=|Xh_0(z)|^2+|Xh_1(z)|^2.
\]
Let us first verify
\begin{equation}\label{eq:M_2}
\widehat{\mathcal{M}}_2^{(1)}(z;X)^2\le |Xh_0(z)|^2+|Xh_1(z)|^2,
\end{equation}
which follows directly from 
 the Cauchy-Schwarz inequality:  for any $f:=\sum^\infty_{j=0}a_jh_j\in{A^2(\Omega)}$ with $\|f\|_2=(\sum^\infty_{j=0}|a_j|^2)^{1/2}=1$,
\begin{eqnarray*}
|Xf(z)|^2 & = & |a_0Xh_0(z)+a_1Xh_1(z)|^2\\
& \leq & (|a_0|^2+|a_1|^2)(|Xh_0(z)|^2+|Xh_1(z)|^2)\\
& \le & |Xh_0(z)|^2+|Xh_1(z)|^2.
\end{eqnarray*}
On the other hand, if we take
\[
a_0=\frac{\overline{Xh_0(z)}}{(|Xh_0(z)|^2+|Xh_1(z)|^2)^{1/2}},\ \ \ a_1=\frac{\overline{Xh_1(z)}}{(|Xh_0(z)|^2+|Xh_1(z)|^2)^{1/2}},
\]
and $a_2=a_3=\cdots=0$, then $\|f\|_2=1$ and
\[
|Xf(z)|^2=|Xh_0(z)|^2+|Xh_1(z)|^2.
\]
This together with \eqref{eq:M_2} complete the proof.
\end{proof}

\begin{remark}
It would be interesting to study the relationship between the upper bound in \eqref{eq:Levi_upper_0}  and the $p-$Bergman metric.
\end{remark}

It seems that the method above does not fit for the case $p>2$ and $p=1$.

\section{Appendix: Proof of Lemma \ref{lm:ineq}}

Similar as \cite{Lindqvist06} or \cite{CZ},  we define 
\begin{eqnarray*}
\eta(t)&:=&|a+t(b-a)|^2=|a|^2+2t\,\mathrm{Re}\left\{\bar{a}(b-a)\right\}+t^2|b-a|^2\\
\kappa(t)&:=&\eta(t)^{q/2}=|a+t(b-a)|^q
\end{eqnarray*}
where $0\le t\le 1$.  
Integral by part gives
\begin{eqnarray}\label{eq:ineq_integral_by_parts}
|b|^q=\kappa(1) &=& \kappa(0)+\kappa'(0)+\int^1_0(1-t)\kappa''(t)dt\nonumber\\
&=& |a|^q+\kappa'(0)+\int^1_0(1-t)\kappa''(t)dt.
\end{eqnarray}
We have
\[
\kappa'(t)=\frac{q}{2}\eta(t)^{(q-2)/2}\eta'(t)=q|a+t(b-a)|^{q-2}\mathrm{Re}\left\{\overline{a+t(b-a)}(b-a)\right\},
\]
so that
\[
\kappa'(0)=q\,\mathrm{Re}\left\{|a|^{q-2}\overline{a}(b-a)\right\}.
\]
Moreover,
\begin{eqnarray*}
\kappa''(t) &=& \frac{q(q-2)}{4}\eta(t)^{(q-4)/2}\eta'(t)^2+\frac{q}{2}\eta(t)^{(q-2)/2}\eta''(t)\\
&=& q(q-2)|a+t(b-a)|^{q-4}\left(\mathrm{Re}\left\{\overline{a+t(b-a)}(b-a)\right\}\right)^2\\
& & +q|a+t(b-a)|^{q-2}|b-a|^2\\
&=& \frac{q(q-2)}{2}|a+t(b-a)|^{q-4}\mathrm{Re}\left\{\overline{(a+t(b-a))^2}(b-a)^2\right\}\\
& & +\frac{q^2}{2}|a+t(b-a)|^{q-2}|b-a|^2\\
&=:& I+II,
\end{eqnarray*}
where the third equality follows from the identity
\begin{equation}\label{eq:Re}
\left(\mathrm{Re}\,c\right)^2=\frac{\mathrm{Re}(c^2)}{2}+\frac{|c|^2}{2},\ \ \ c\in\mathbb{C}.
\end{equation}

Let us first consider the term $I$.  \eqref{eq:ineq_leq2_2} and \eqref{eq:ineq_geq2_2} yield 
\begin{equation}\label{eq:cd}
\left||d|^{q/2-2}\overline{d}-|c|^{q/2-2}\overline{c}\right| \leq \begin{cases}
2^{2-q/2}|d-c|^{q/2-1},\ \ \ &2<q<4,\\
(q/2-1)\left(|c|+|d-c|\right)^{q/2-2}|d-c|,\ \ \ &q\ge 4.
\end{cases}
\end{equation}
for $c,d\in\mathbb{C}$. Substitute $c=a^2$ and $d=(a+t(b-a))^2$ into \eqref{eq:cd} for $2<q<4$, we obtain
\begin{eqnarray*}
& & \left||a+t(b-a)|^{q-4}\overline{(a+t(b-a))^2}-|a|^{q-4}\overline{a}^2\right|\\
&\leq& 2^{2-q/2}\left|2ta(b-a)+t^2(b-a)^2\right|^{q/2-1}\\
&\leq& 2^{2-q/2}\left(2|a|+|b-a|\right)^{q/2-1}|b-a|^{q/2-1}\\
&\leq& 2\left(|a|+|b-a|\right)^{q/2-1}|b-a|^{q/2-1}\\
&\leq& 2\left(|a|^{q/2-1}+|b-a|^{q/2-1}\right)|b-a|^{q/2-1}
\end{eqnarray*}
since $0\leq{t}\leq1$. Similarly,  for $q\geq4$ we have
\begin{eqnarray*}
& & \left||a+t(b-a)|^{q-4}\overline{(a+t(b-a))^2}-|a|^{q-4}\overline{a}^2\right|\\
&\leq& (q/2-1)\left(|a|+|b-a|\right)^{q-4}\left|2ta(b-a)+t^2(b-a)^2\right|\\
&\leq& (q-2)\left(|a|+|b-a|\right)^{q-3}|b-a|\\
&\leq& (q-2)2^{q-4}\left(|a|^{q-3}+|b-a|^{q-3}\right)|b-a|.
\end{eqnarray*}
Thus
\begin{eqnarray}\label{eq:I}
I &=& \frac{q(q-2)}{2}\mathrm{Re}\left\{|a+t(b-a)|^{q-4}\overline{(a+t(b-a))^2}\cdot(b-a)^2\right\}\nonumber\\
&=& \frac{q(q-2)}{2}\mathrm{Re}\left\{|a|^{q-4}\overline{a}^2(b-a)^2\right\}+F_q(t,a,b)\nonumber\\
&=& \frac{q(q-2)}{2}|a|^{q-4}\mathrm{Re}\left\{\overline{a}^2(b-a)^2\right\}+F_q(t,a,b)
\end{eqnarray}
where
\begin{equation}\label{eq:F_q}
|F_q(t,a,b)|\leq\begin{cases}
q(q-2)\left(|a|^{q/2-1}+|b-a|^{q/2-1}\right)|b-a|^{q/2+1},\ \ \ &2<q<4,\\
q(q-2)^22^{q-5}\left(|a|^{q-3}+|b-a|^{q-3}\right)|b-a|^3,\ \ \ &q\geq4.
\end{cases}
\end{equation}

It remains to deal with the term $II$. Let us first consider the case $2<q\leq{3}$.  Use \eqref{eq:ineq_elementary}, we obtain
\[
\left||a+t(b-a)|^{q-2}-|a|^{q-2}\right| \leq t^{q-2}|b-a|^{q-2}\leq|b-a|^{q-2}.
\]
Thus
\begin{eqnarray}\label{eq:II}
II &=& \frac{q^2}{2}|a+t(b-a)|^{q-2}|b-a|^2\nonumber\\
&=& \frac{q^2}{2}|a|^{q-2}|b-a|^2+G_q(t,a,b),
\end{eqnarray}
where
\begin{equation}\label{eq:G_q_1}
|G_q(t,a,b)|\leq\frac{q^2}{2}|b-a|^q,\ \ \ 2<q\leq3.
\end{equation}
When $q>3$, we infer from \eqref{eq:ineq_geq1} and \eqref{eq:ineq_elementary} that
\begin{eqnarray*}
|a+t(b-a)|^{q-2} &\leq& |a|^{q-2}+(q-2)t|a+t(b-a)|^{q-3}|b-a|\\
&\leq& |a|^{q-2}+C_q(|a|^{q-3}+|b-a|^{q-3})|b-a|,
\end{eqnarray*}
where $C_q:=(q-2)\max\{1,2^{q-4}\}$. Again by \eqref{eq:ineq_geq1}, we have
\begin{eqnarray*}
|a+t(b-a)|^{q-2} &\geq& |a|^{q-2}-(q-2)t|a|^{q-3}|b-a|\\
&\geq& |a|^{q-2}-(q-2)(|a|^{q-3}+|b-a|^{q-3})|b-a|,
\end{eqnarray*}
so that \eqref{eq:II} holds with
\begin{equation}\label{eq:G_q_2}
|G_q(t,a,b)|\leq\frac{1}{2}\max\left\{q^2C_q,q^2(q-2)\right\}(|a|^{q-3}+|b-a|^{q-3})|b-a|^3.
\end{equation}
This together with \eqref{eq:I} give
\[
\int^1_0(1-t)\kappa''(t)dt = \frac{q^2}{4}|a|^{q-2}|b-a|^2+\frac{q(q-2)}{4}|a|^{q-4}\mathrm{Re}\left\{\overline{a}^2(b-a)^2\right\}+R_q(a,b),
\]
where
\[
R_q(a,b)=\int_0^1 (1-t) \left( F_q(t,a,b)+G_q(t,a,b)\right)dt.
\]
This together with \eqref{eq:ineq_integral_by_parts},   \eqref{eq:F_q}, \eqref{eq:G_q_1} and \eqref{eq:G_q_2} concludes the proof.

{\bf Acknowledgements.} The authors would like to thank Liyou Zhang for bringing our attention to \cite{LiYinji} and the referee for valuable comments and suggestions.

{\bf Data Availability Statement.} The paper has no associated data.

\end{document}